\documentclass{amsart}

\usepackage[latin1]{inputenc}
\theoremstyle{plain}
\newtheorem{theorem}{Theorem}[section]
\newtheorem{corollary}[theorem]{Corollary}
\newtheorem{lemma}[theorem]{Lemma}
\newtheorem{proposition}[theorem]{Proposition}

\theoremstyle{definition}
\newtheorem{definition}[theorem]{Definition}

\theoremstyle{remark}
\newtheorem{remark}[theorem]{Remark}

\newcommand{\A}{\mathcal{A}}
\newcommand{\D}{\mathcal{D}}
\newcommand{\Oh}{\mathcal{O}}
\newcommand{\C}{\mathbb{C}}

\newcommand{\PG}{\mathbb{P}^1}
\newcommand{\PR}{\mathbb{P}}
\newcommand{\U}{\mathcal{U}}
\newcommand{\K}{\mathcal{K}}
\newcommand{\la}{\alpha}
\newcommand{\lb}{\beta}
\newcommand{\lc}{\gamma}

\newcommand{\Harm}{\mathcal{H}}

\newcommand{\ovl}{\overline}
\newcommand{\dbar}{\bar \partial}
\newcommand{\dl}{ \partial}

\newcommand{\dbarstar}{\bar \partial^*}

\renewcommand\>{\rangle}

\newcommand{\X}{\mathcal{X}}

\newcommand{\M}{\mathcal{M}_g}
\newcommand{\Hbn}{\mathcal{H}^{n,b}}
\newcommand{\Hx}{\mathcal{H}_X}

\newcommand{\G}{G^{WP}}
\newcommand{\jbar}{\ovl{\jmath}}

\newcommand{\lbar}{\ovl{l}}

\newcommand{\wbar}{\ovl{w}}
\newcommand{\zbar}{\ovl{z}}
\newcommand{\kbar}{\ovl{k}}
\newcommand{\wY}{\omega_Y}
\newcommand{\z}{\zeta_z^w}
\newcommand{\zb}{\zeta_{\ovl{z}}^{\ovl{w}}} 
\newcommand{\hz}{h_{z\ovl{z}}}
\newcommand{\gz}{g_{z\ovl{z}}}
\newcommand{\hw}{h_{w\ovl{w}}}
\newcommand{\ginv}{g^{\zbar z}}
\newcommand{\Kw}{K_{w\wbar}}
\newcommand{\vbar}{\ovl{v}}
\newcommand{\sbar}{\ovl{s}}

\newcommand{\tz}{\tilde{z}}
\newcommand{\tg}{\tilde{g}}
\newcommand{\tbeta}{\tilde{\beta}}
\newcommand{\tw}{\tilde{w}}

\newcommand{\id}{\operatorname{id}}

\newcommand{\laplacedbar}{\Box_{\dbar}} 
\newcommand{\Hbar}{H_{\dbar}}
\newcommand{\Gbar}{G_{\dbar}}

\numberwithin{equation}{section}

\begin{document}

\title{K\"ahler geometry on Hurwitz spaces}

\author{Philipp Naumann}
\address{Fachbereich Mathematik und Informatik,
Philipps-Universit\"at Marburg, Lahnberge, Hans-Meerwein-Straße, D-35032
Marburg, Germany}
\email{naumann@mathematik.uni-marburg.de}



\date{}

\begin{abstract}
We study the K\"ahler geometry of the classical Hurwitz space $\Hbn$ of simple branched coverings of the Riemann sphere $\mathbb{P}^1$ by compact hyperbolic Riemann surfaces. A generalized Weil-Petersson metric on the Hurwitz space was recently introduced in \cite{ABS15}. Deformations of simple branched coverings fit into the more general framework of Horikawa's deformation theory of holomorphic maps, which we equip with distinguished representatives in the presence of hermitian metrics. In the article we will investigate the curvature of the generalized Weil-Petersson K\"ahler metric on the Hurwitz space. 
\end{abstract}

\maketitle


\section{Introduction}
\subsection{Hurwitz spaces}
The classical Hurwitz space $\Hbn$ parametrizes isomorphism classes of simple branched coverings
$$
f: X \to \PG
$$
of degree $n$ with $b$ branching points, where $X$ is a compact hyperbolic Riemann surface. These spaces first appeared in the works of Clebsch \cite{Cl72} and Hurwitz \cite{Hu91}, where they showed that $\Hbn$ is connected. Further investigations using the language of modern algebraic geometry were made by Fulton \cite{Fu69} as well as Harris and Mumford \cite{HM82}. Originally, Hurwitz spaces were used as auxiliary objects to study the moduli space $\M$ of compact Riemann surfaces of genus $g>1$. For example, the existence of a natural holomorphic map
$$
\Hbn \to \M,
$$
which is surjective for $n>g$ (\cite{Sev21}),  gives a more elementary proof of the fact that $\M$ is irreducible.
The algebraic geometers studied the geometry of the Hurwitz space by means of the finite topological covering 
$$
\Hbn \to \operatorname{Sym}^b(\PG) \setminus \Delta.
$$
In this article, we take a deformation theoretic point of view and study the Hurwitz space by means of its universal family
$$
\X \to \PG \times \Hbn.
$$
Inspired by the theory of families of K\"ahler-Einstein manifolds in higher dimensions, we equip the fibers $X_s$ with hyperbolic metrics and $\PG$ with the Fubini-Study metric. This allows us to introduce a generalized Weil-Petersson metric on $\Hbn$, which turned out to be K\"ahler (\cite{ABS15}). The metric reflects the variation of the meromorphic maps as well as the variation of the underlying complex structures. Since the methods of K\"ahler geometry are now available, we are able to study the differential geometric properties of the Hurwitz space. The main focus lies on the curvature of the Weil-Petersson metric.

\subsection{Differential geometric setup and statement of results}
The geometric study of the Hurwitz space using methods from differential geometry starts with the work of Axelsson, Biswas and Schumacher in \cite{ABS15}, where they had a deformation theoretic point of view. We first recall their setup and definitions.

We consider a holomorphic family of coverings
$$
(\beta,f): \X \to Y \times S
$$ 
of compact hyperbolic Riemann surfaces $X_s$ and a fixed compact Riemann surface $Y$. (Note that this leads to generalized Hurwitz stacks $\Hbn(Y)$, see \cite{HGS02}).
 Choose local coordinates $(z,s=s^1,\ldots,s^r)$ on $\X$ and $w$ on 
$Y$ such that
$$
\beta(z,s)=w \mbox{ and } f(z,s)=s. 
$$
The fibers $X_s$ carry unique hyperbolic metrics 
$$
\omega_{X_s}=\sqrt{-1}\, g(z,s)\, dz \wedge d\zbar 
$$
of constant Ricci curvature $-1$. 
Let $v_s$ be the \emph{horizontal lift} of a tangent vector $\dl_s$ on $S$ at $s$ introduced by Schumacher in \cite{Sch93}, which are special \emph{canonical lifts} in the sense of \cite{Siu86}. The harmonic representative of the Kodaira-Spencer map $\rho_s: T_{S,s} \to H^1(X_s,T_{X_s})$ is given by
$$
\mu_s = (\dbar v_s)\vert_{X_s}.
$$
Set
$$
\varphi := \<v_s,v_s\>_{\omega_{\X}}\quad \mbox{ where } \quad \omega_{\X} := \sqrt{-1} \dl \dbar \log g(z,s).
$$
According to \cite{Sch93} we have
$$
(\Box_{\omega_s}+1)\varphi= \| \mu_s\|^2.
$$
Define
$$
u_s := \beta_*(v_s)=\beta_*(\dl_s + a^z_s \dl_z )=(\xi^w_s + a^z_s \z)\dl_w,
$$
where
$$
\xi^w_s = \frac{\dl \beta(z,s)}{\dl s} \quad
\mbox{ and } \quad
\z =\frac{\dl \beta(z,s)}{\dl z}.
$$
In this article, we give a more conceptual definition of the vector fields $u_s$. Refining Horikawa's theory of deformations of holomorphic maps in the presence of hermitian metrics, the vector fields $u_s\vert_{X_s} \in A^{0,0}(X_s,\beta_s^*T_Y)$ turn out to be \emph{generalized harmonic representatives} of the characteristic map (Kodaira-Spencer map for deformations of maps) $\tau_s: T_{S,s} \to H^0(X_s,N_{\beta_s})$. Here $N_{\beta_s}$ is the normal sheaf of the map $\beta_s: X_s \to Y$. Moreover, we will apply this general results to the case of coverings of Riemann surfaces. Remembering that the branching points give local coordinates on the Hurwitz space, we can give an answer to the question which infinitesimal movements of branching points actually change the complex structure of the overlying hyperbolic Riemann surface. After providing the reader with the basic properties of Hurwitz spaces in the next section, these deformation theoretic aspects are the content of section 3. 

After introducing the necessary objects, the Weil-Petersson metric can be defined as follows:
\begin{definition}\cite{ABS15}
Let $\omega_Y$ be a metric on $Y$ of constant Ricci curvature equal to $\epsilon =0$ or $\pm1$ depending on its genus.
The Weil-Petersson inner product on the tangent space $T_S$ of the base $S$ is defined by its norm
\begin{eqnarray*}
\|\dl_s\|^2_{WP} &:=& \G_{s\sbar}(s) \\ 
&:=& \int_{X_s}{(\Box_{\omega_s}+1)^{-1}(\|\mu_s\|^2(z,s)) \,\beta^*_s\omega_Y}\\
&+& \int_{X_s}{\|u_s\|^2(z,s)\; \omega_{X_s}}\\
\end{eqnarray*}
\end{definition}

\begin{proposition}[\cite{ABS15}]
The product is positive definite if the family $(\beta,f): \X \to Y \times S$ is effectively parametrized. Furthermore, 
the Weil-Petersson form satisfies the fiber integral formula
$$
\omega^{WP} = \int_{\X/S}{\omega_{\X}\wedge \beta^*\omega_Y}
$$
In particular, the Weil-Petersson form is K\"ahler.
\end{proposition}

In section 4, we will study the Weil-Petersson metric on the base of a family of coverings. We provide a list of useful identities for later computations. Moreover, we will reprove the K\"ahler identity by a direct computation.
 
In section 5, we consider the universal family $(\beta,f): \X \to \PG \times \Hbn$.
Under the assumption $b>4g-4$, the coherent sheaf $f_*\beta^*T_{\PG}$ is locally free and a holomorphic subbundle of the tangent bundle $T_{\Hbn}$ on the Hurwitz space. As our main result, we obtain
\begin{theorem}
The induced Weil-Petersson metric on the subbundle $f_*\beta^*T_{\PG}$ (which is fiberwise the natural $L^2$-metric on $H^0(X_s,\beta_s^*T_{\PG})$) has curvature
\begin{eqnarray*}
R^{WP}_{i\jbar k\lbar}(s)=&-& \int_{X_s}{D_k(u_i^w)D_{\lbar}(u^{\wbar}_{\jbar})\hw \; g\,dA} \\
&+&  \int_{X_s}{u_i^w u_{\jbar}^{\wbar}(\xi^w_k + a_k^z \z)(\xi^{\wbar}_{\lbar} + a_{\lbar}^{\zbar}\zb)\hw^2 \; g\,dA} \\
&+&  \int_{X_{s}}{\Box\varphi_{k\lbar} u_{i}^w u_{\jbar}^{\wbar}\hw \; g\,dA},
\end{eqnarray*}
where $u_1,\ldots,u_r \in \Gamma(U,f_*\beta^*T_{\PG})$ are local holomorphic sections,
$D_k$ is the covariant derivative in the direction of the horizontal lift $v_k$,
$\omega_Y = \sqrt{-1}\, h_{w\wbar}\, dw \wedge d\wbar$,
$dA = \sqrt{-1} \, dz\wedge d\zbar$ and $\omega_{X_s}=g\,dA$.
\end{theorem}

The first summand can be split up into four terms such that the explicit derivatives in base direction completely disappear. 
Note that we are dealing with a \emph{coupled situation} where the complex structure on $X_s$ as well as the hermitian bundle $(\beta_s^*T_{\PG},\beta_s^*h)$ varies. The difficulty in the computation of the curvature arises from the fact that we have two competing hermitian metrics $\omega_{X_s}$ and $\omega_{\PG}$ both contributing to the Weil-Petersson metric. But there is no intimate relation between those two metrics. 

In the situation $b>4g-4,$ the canonical map $\Hbn \to \M$ is a submersion onto the smooth open subset of the moduli space belonging to curves having trivial automorphism group. Writing $\M^0$ for this part and replacing $\Hbn$ by the open subset $\Hbn_0$ lying over $\M^0$, we can write
$$      
\K^{-1}_{\Hbn_0/\M^0} \cong \det(f_*\beta^*T_{\PG}),
$$
where we have restricted the family to $\Hbn_0$.
We denote the fiber of a class $[X]$ under $\Hbn_0 \to \M^0$ by $\Hx$.  
Thus, $\Hx$ is a complex submanifold of $\Hbn$ of dimension $(2n-g+1)$.
We have the relation
$$
f_*\beta^*T_{\PG}\vert_{\Hx} = T_{\Hx}.
$$
\begin{corollary}
The curvature of the restricted Weil-Petersson metric is given by
\begin{eqnarray*}
R_{i\jbar k \lbar}(s) = &-& \int{\Gbar \left( \psi_{ik} \right) \cdot \psi_{\jbar\lbar}\; g\, dA} \\
                          &+&  \int{(\xi_i \cdot \xi_{\jbar}) \, (\xi_k \cdot \xi_{\lbar}) \; g\, dA}  
\end{eqnarray*}
where
$$
\psi_{ik}\, d\zbar \otimes \dl_w := \left( \xi^w_i \xi^w_k \zb \hw \right) d\zbar \otimes \dl_w. 
$$
\end{corollary}
Although stated as a corollary, we will provide a separate proof of this statement.

The results in this article are from the authors dissertation \cite{Na16}.

\section{Construction of Hurwitz spaces and first properties}
It is very classical to study compact Riemann surfaces by branched coverings of the Riemann sphere. The generic case is one of a \emph{simple} branched covering with only two sheets meeting over each branch point. The number $b$ of branch points for a simple covering of degree $n$ with total space of genus $g$ is given by the Riemann-Hurwitz formula
$$
b=2n+2g-2
$$ 
The classical \emph{Hurwitz space} $\Hbn$ is the set of equivalence classes of simple branched coverings $f: X \to \PG$ of compact Riemann surfaces $X$ of degree $n$ with $b$ branch points, $(n,b)$-covering for short, where $X$ is considered to be hyperbolic. Here we say that two coverings $f: X \to \PG$ and $g: X' \to \PG$ are equivalent iff there is an automorphism 
$\varphi: X \to X'$ such that $g \circ \varphi=f$. A covering map is given by the $b$-tuple of its branch points and a finite number of certain monodromy data. This gives a map from the Hurwitz space to the set of unordered $b$-tuples of $b$ distinct points of $\PG$, which can be identified with $\mathbb{P}^b \setminus \Delta$ where $\Delta$ is the discrimination locus.  We can put a topology on $\Hbn$ such that this map becomes a finite unbranched topological covering: Given a $(n,b)$-covering $f: X \to \PG$ with branching points $B=(y_1,\ldots, y_b)$, we can move these points inside disjoint open discs around the points $y_j$ by a homeomorphism of $\PG$ to obtain a new covering of degree $n$ with $b$ branching points, which is a priori just continuous. After removing the branching points and the corresponding fibers of this map,  we can complete it by the Riemann Existence Theorem to a holomorphic $(n,b)$-covering $f_{B'}: X \to \PG$, where the complex structure on the surface $X$ may change. Using the correspondence between topological coverings and subgroups of the fundamental group, we see that $f_{B'}$ only depends on $f$ and the new positions $B'=(y_1',\ldots,y_b')$ of the branch points and not on the chosen homeomorphism.  For that reason, $\Hbn$ can be equipped with a complex structure, so that the Hurwitz space becomes an affine complex manifold of dimension $b$. Relying on calculations of Clebsch \cite{Cl72}, Hurwitz showed in \cite{Hu91} that $\Hbn$ is connected. Due to Fulton and his fundamental article \cite{Fu69}, there is a universal family for $n>2$:
$$
(\beta,f): \X \to \PG \times \Hbn.
$$
This family can be constructed analytically: First using the process of constructing open neighborhoods in $\Hbn$ just described, we can construct them locally. By fact that there are no non-trivial  automorphisms of simple branched $(n,b)$-coverings for $n>2$, it follows that the isomorphisms between two equivalent $(n,b)$-coverings are unique. (An automorphism $\varphi: X \to X$ of a covering with $\varphi \neq \operatorname{id}$ would yield an unbranched covering $X/\varphi \to \PG$ and thus $X/\varphi \cong \PG$, which means $n=2$.) Therefore, the local families can be glued together to give a global family over $\Hbn$. 

Now we discuss briefly the Kobayashi hyperbolicity of the Hurwitz space, since we are especially interested in differential geometric properties of this space. For a simple branched $(n,b)$-covering $\beta_0: X \to \PG$ and an automorphism 
$\la: \PG \to \PG$, the map $\la \circ \lb_0: X \to \PG$ is again a simple branched $(n,b)$-covering, which cannot be equivalent to 
$\lb_0$ for $\la \neq \operatorname{id}$.  Therefore, the $3$-dimensional Lie-group 
$\operatorname{Aut}(\PG)=\operatorname{PGL}(2)$ acts on the Hurwitz space $\Hbn$. Hence, we have that the Hurwitz space 
$\Hbn$ is not hyperbolic (in the sense of Kobayashi). But what if we eliminate the action of $\operatorname{Aut}(\PG)$? We move our point of view to a different category of coverings: Two branched coverings $f: X \to \PG$ and $f': X' \to \PG$ are considered to be equivalent iff there exist biholomorphic maps $\varphi: X \to X'$ and $\psi: \PG \to \PG$ such that $f' \circ \varphi = \psi \circ f$. Now consider the branch points $P_1, \ldots,P_b$ as an ordered $b$-tuple. Because the M\"obius group acts exactly threefold transitively on $\PG$, we can reach $P_{b-2}=0, P_{b-1}=1$ and $P_b=\infty.$ Now we construct the so called \emph{reduced Hurwitz space} $\Hbn_{red}$ as a finite unbranched topological covering of 
$$
[(\PR^1)^b \setminus \bigcup_{i<j}\Delta_{ij}]/\operatorname{PGL(2)}\cong (\PG \setminus \{0,1,\infty\})^{b-3} \setminus \Delta_{b-3},
$$ 
where $\Delta_{b-3}$ is the weak diagonal. This is the point of view in \cite{HM82}. Now the space $\PG \setminus \{0,1,\infty\}$ is hyperbolic as well as $ (\PG \setminus \{0,1,\infty\})^{b-3} \setminus \Delta_{b-3}$ and we get that the reduced Hurwitz space 
$\Hbn_{red}$ is hyperbolic. This space is now a complex manifold of dimension $2n+2g-5$, compare \cite[Th. 3.4.17]{Na79}.
The notion \emph{reduced Hurwitz space} first appeared in the arithmetic theory (\cite{DF99,BF02,Ca08}). But there the branch points are unordered. Sometimes one studies the Hurwitz space by definition as reduced modulo $\operatorname{Aut}(\PG)$, see \cite{Pa13}.

Originally, the Hurwitz space was used as an auxiliary object to study the moduli space $\M$ of compact Riemann surfaces of genus $g>1$. The existence of the universal family 
$\X \to \Hbn$ gives a natural holomorphic map
$$
\Hbn \to \M
$$
by mapping a simple covering $X \to \PG$ to the isomorphism class of $X$. Using appropriate linear systems and Riemann-Roch
(see \cite{Fu69}), one can show that this map is even surjective for $n>g$. By studying the fibers of this map, which have dimension $2n-g+1$, Riemann obtained for the dimension of $\M$ in his famous moduli count \cite{Ri57}:
$$
\dim \M = b - (2n-g+1) = (2n+2g-2)-(2n-g+1)=3g-3.
$$
We will compute the curvature of the fibers of the map $\Hbn \to \M$.

\section{Methods of deformation theory}
\subsection{Horikawa's theory}
In this section, we introduce the Kodaira-Spencer map for deformations of holomorphic maps with fixed target. This goes back to Horikawa (\cite{Ho73,Ho74}). For the classical theory of deformations of complex structures, which is not recalled here, we refer to \cite{Ma05}. Let $Y$ be a fixed compact complex manifold. 
\begin{definition}
A \emph{family of holomorphic maps to $Y$} consists of a family $(\X,p,S)=(X_s)_{s \in S}$ of compact complex manifolds parametrized by a connected complex manifold $S$ together with a holomorphic map $F: \X \to Y$. We collect these data in the quadrupel $(\X,p,S,F)$. We set $f_s=F|_{X_s}: X_s \to Y$ and also denote the family by $(X_s,f_s)_{s \in S}$. 
\end{definition}
Now let $(\X,p,S,F)$ be such a family of holomorphic maps to $Y$ and $s_0 \in S$ a fixed point. We set $X=X_{s_0}$ and 
$f:=F|_X: X \to Y$. We have an exact sequence of coherent sheaves 
\begin{equation}
\label{exse}
0 \to T_{X/Y} \to T_X \xrightarrow{d f} f^* T_Y \xrightarrow{P} N_f \to 0,
\end{equation}
where $N_f$ is the normal sheaf on $f$, i.e. the cokernel of $df$.
After restricting to a neighborhood of $s_0 \in S$ if necessary, we  have the following setting:
\begin{itemize}
\item[(i)] $S$ is an open subset in $\C^r$ with coordinates $s=(s^1,\ldots,s^r)$ and $s_0=(0,\ldots,0)$.
\item[(ii)] $\X$  is covered by a finite number of Stein coordinate neighborhoods $\U_i$ together with  coordinates 
$(z_i,s)=(z_i^1,\ldots,z_i^n,s^1,\ldots,s^r)$ such that $p(z_i,s)=s$.
\item[(iii)] $Y$ is covered by a finite subset of Stein coordinate neighborhoods $V_i$ with local coordinates $w_i=(w_i^1,\ldots,w_i^m)$ such that $F(\U_i) \subset V_i$ and $F$ is given by
$$
w_i=F(z_i,s).
$$
\item[(iv)] On $\U_i \cap \U_j$, we have holomorphic transition functions
$$
z_i=f_{ij}(z_j,s).
$$
\item[(v)] On $V_i \cap V_j$, we have holomorphic transition functions
$$
w_i=g_{ij}(w_j).
$$
\end{itemize}
Then we get the relation
\begin{equation}
\label{vertr}
F(f_{ij}(z_j,s),s)=g_{ij}(F(z_j,s))
\end{equation}
For any tangent vector $\dl/\dl s \in T_{S,0}$ we set
$$
\tau_i:=\sum_{\lc=1}^m{\frac{\dl F_i^{\lc}}{\dl s}{\Big |}_{s=0}\frac{\dl}{\dl w_i^{\lc}}} \in \Gamma(U_i,f^*T_Y) \quad (U_i=\U_i \cap \X).
$$
Then form equality \ref{vertr} we infer that
$$
\tau_j - \tau_i = f_*\left( \sum_{\sigma}{\frac{\dl f_{ij}^{\sigma}}{\dl s}\frac{\dl}{\dl z_i^{\sigma}}}\right),
$$
where 
$$
\theta_{ij}:=\sum_{\sigma}{\frac{\dl f_{ij}^{\sigma}}{\dl s}{\Big |}_{s=0}\frac{\dl}{\dl z_i^{\sigma}}}
$$
is a representative of the Kodaira-Spencer class at $s=0$ of the family $p: \X \to S$. Therefore, the $0$-cochain $(P\tau_i)$ defines an element of $H^0(X,N_f)$. Thus we can define a linear map
$$
\tau: T_{S,0} \to H^0(X,N_f),
$$
the so called \emph{characteristic map} for the family of holomorphic maps. The space $H^0(X,N_f)$ describes the infinitesimal deformations of the holomorphic map $f$, see \cite{Se06}. To get a better description of the elements of this space, we define
\begin{definition}
$$
D_{X/Y} := \frac{\{ (\tau,\theta) \in C^0(\U,f^*T_Y)\times Z^1(\U,T_X) \, : \, \delta \tau = f_* \theta \} }{\{ (f_* \sigma,\delta \sigma) \, : \, \sigma \in C^0(\U,T_X) \} }
$$
\end{definition}
\begin{lemma}
\
\begin{itemize}
\item[(i)] $D_{X/Y}$ does not depend on the choice of the Stein covering.
\item[(ii)] $D_{X/Y}$ is a finite dimensional vector space.
\item[(iii)] We have the following two exact sequences:
\begin{equation}
\label{seq1a}
H^0(X,T_X) \xrightarrow{df} H^0(X,f^*T_Y) \to D_{X/Y} \to H^1(X,T_X) \to H^1(X,f^*T_Y),
\end{equation}
\begin{equation}
\label{seq2}
0 \to H^1(X,T_{X/Y}) \to D_{X/Y} \to H^0(X,N_f) \to H^2(X,T_{X/Y}).
\end{equation}
\end{itemize}
\end{lemma}
For a proof see \cite{Ho73}. 
\begin{corollary}
\begin{itemize}
\item[(i)] If $f$ is non-degenerante (i.e. $T_{X/Y}=0$), we have $D_{X/Y}\cong H^0(X,N_f)$.
\item[(ii)] If $f$ is smooth, we have $D_{X/Y} \cong H^1{X,T_{X/Y}}$.
\end{itemize}
\end{corollary}
There is also a Dolbeault-type description of the space $D_{X/Y}$ also due to Horikawa:
\begin{proposition}
$$
D_{X/Y} \cong \frac{ \{(\xi,\vartheta) \in A^{0,0}(X,f^*T_Y)\times A^{0,1}(T_X) \, : \, \dbar \xi = f_* \vartheta, \dbar \vartheta =0\} }{ \{ (f_*\zeta,\dbar \zeta): \zeta \in A^{0,0}(X,T_X) \} } =:D'_{X/Y}
$$
\end{proposition}
\begin{proof}
Let $(\tau,\theta) \in C^0(\U,f^*T_Y) \times Z^1(\U,T_X)$ be a representative of a class in $D_{X/Y}$. We set $\tau=(\tau_i)$ and 
$\theta=(\theta_{ij})$. We choose $\eta_i \in \Gamma(U_i,\A^{0,0}(T_X))$ with $-\theta_{ij}=\eta_j - \eta_i$ in $U_{ij}$. Then we define  $\vartheta \in A^{0,1}(X,T_X)$ by
$$
\vartheta = \dbar \eta_i \quad \mbox{ on } U_i,
$$
i.e. $\vartheta$ is a Dolbeault representative of the class $[\theta]$. 
On the other hand we have $\tau_j - \tau_i = -f_*\eta_j + f_* \eta_i$ in $U_{ij}$. Thus we define $\xi \in A^{0,0}(f^*T_Y)$ by 
$$
\xi = \tau_i + f_*\eta_i \quad \mbox{ on } U_i.
$$
Both objects are related by
$$
\dbar \xi = f_* \vartheta.
$$	
If we have a second cochain $\eta'=(\eta_i') \in C^0(\U,\A^{0,0}(T_X))$ with $-\theta_{ij}=\eta_j - \eta_i$, then $\eta$ and $\eta'$  differ by a $\zeta \in A^{0,0}(T_X)$. The forms $\vartheta$ and $\xi$ then become $\vartheta + \dbar \zeta$ and $\xi + f_*\zeta$. If we start with the representative $(f_*\sigma,\delta \sigma)$ with 
$\sigma \in C^0(\U,T_X)$ belonging to the trivial class in $D_{X/Y}$, we can choose $\eta_i=-\sigma|_{U_i}$ and get 
$\vartheta=\dbar(-\sigma)=0$ as well as 
$\xi=f_* \sigma + f_*(-\sigma)=0.$ Therefore, the class $(\xi,\vartheta)$ does not depend on the chosen representative. 

Conversely, if $(\xi,\vartheta)$ with $\dbar \xi = f_* \vartheta$ and $\dbar \vartheta=0$ are given, we find $\eta_i \in \Gamma(U_i,\A^{0,0}(T_X))$ such that
$\vartheta=\dbar \eta_i$ in $U_i$. We set $\theta_{ij}=-\eta_j + \eta_i$ in $U_{ij}$ and $\tau_i=\xi - f_*\eta_i$ in $U_i$. Then $\dbar \theta_{ij}=\dbar \tau_i=0,$ i.e. 
$\theta=(\theta_{ij}) \in Z^1(\U,T_X)$ and $\tau=(\tau_i) \in C^0(\U,f^*T_Y)$. Furthermore, the relation $\tau_j-\tau_i=f_*\theta$ holds. If we have a second cochain $\eta'=(\eta_i') \in C^0(\U,\A^{0,0}(T_X))$ such that $\vartheta=\dbar \eta_i'$ in $U_i$, then 
$\eta$ and $\eta'$ differ by $\sigma \in C^0(\U,T_X)$; $\theta$ and $\tau$ become 
$\theta + \delta \sigma$ and $\tau + f_*\sigma$ respectively. Hence, the class $(\tau,\theta)$ in $D_{X/Y}$ is well-defined. Furthermore, we see that a pair of the form $(f_*\zeta,\dbar \zeta)$ with $\zeta \in A^{0,0}(T_X)$  is mapped on the trivial class $(0,0)$. This gives the desired isomorphism.
\end{proof}
After recalling the work of Horikawa, we would like to give a constructive description of how to produce representatives of the characteristic map in $D'_{X/Y}$:
\begin{proposition}
\label{diffbarehochhebung}
Let $(\X,p,S,F)$ be a family of non-degenerate holomorphic maps to $Y$. For a point $s_0 \in S$, we set $X=X_{s_0}$ and 
$f=F|_X: X \to Y$. Let $\dl/\dl s \in T_{S,s_0}$ be a tangent vector. Extend this vector to a local holomorphic vector field in a neighborhood
$V \subset S$ of $s_0 \in S$. Let $\chi \in \Gamma(p^{-1}(V),\A^{0,0}(T_{\X}))$ be a differentiable lift of this vector field. Then the class $\tau_{s_0}(\dl/\dl s)$ is represented in $D'_{X/Y}$ by the pair $(F_*(\chi)|_X,\dbar(\chi)|_X)$.   
\end{proposition} 
\begin{proof}
Let $\U=(\U_i)$ be a locally finite Stein open covering of $p^{-1}(V)$ by coordinate neighborhoods in which $p$ is just a projection.
Then $\tau(\dl / \dl s) \in D_{X/Y}$ is given by the pair $(\tau,\theta) \in C^0(\U,f^*T_Y)\times Z^1(\U,T_X)$, where  $\delta \tau = f_* \theta$. In local coordinates,
$\tau=(\tau_i)$ has the form
$$
\tau_i = \sum_{\lc}{\frac{\dl F_i^{\lc}}{\dl s}}{\Big |}_{s=s_0} \frac{\dl}{\dl w_i^{\lc}} \in \Gamma(U_i,f^*T_Y) \quad (U_i=\U_i \cap X)
$$
We note that $\tau_i=F_*(\dl / \dl s)|_X$, where we view $s$ as a local coordinate on $\U_i$. 
The cocycle $\theta=(\theta_{ij})$ defined by
$$
\theta_{ij}=\sum_{\sigma}{\frac{\dl f_{ij}^{\sigma}}{\dl s}{\Big |}_{s=s_0} \frac{\dl}{\dl z_i^{\sigma}}}
$$
is representing the Kodaira-Spencer class $\rho(\dl / \dl s)$. We already know that $\dbar(\chi)|_X$ is a Dolbeault representative of $\theta$ (see \cite{Ma05}). Write $\chi$ locally as
$$
\chi = \sum_{\la}{\chi_i^{\la}(z_i,s)\frac{\dl}{\dl z_i^{\la}}+ \frac{\dl}{\dl s}} \mbox{  in } \U_i. 
$$
Since $\dbar(\chi - \dl/\dl s)= \dbar(\chi)$, the $0$-cochain $\eta=(\eta_i)$ defined by
$$
\eta_i = \sum_{\la}{\chi_i^{\la}(z_i,s)\frac{\dl}{\dl z_i^{\la}}} \in \Gamma(U_i,\D(T_{\X/S}))
$$
yields a differentiable splitting of $\theta$. After the proof of the preceding lemma, $\tau(\dl / \dl s) \in D'_{X/Y}$  is thus given by the pair
$(\xi,\theta)$, where 
$$
\xi = \tau_i + f_*\eta_i \quad \mbox{ in } U_i.
$$  
From this it finally follows $\xi=F_*(\chi)|_X$.
\end{proof} 

\subsection{Refining Horikawa's theory in the presence of hermitian metrics}
In this subsection, we consider more specifically families of non-degenerate holomorphic maps $f_s: X_s \to Y$, where the fibers 
$X_s$ are compact hermitian manifolds. In this case, any Dolbeault class in $H^1(X_s,T_{X_s})$ has a unique harmonic representative (with respect to $\laplacedbar$). We fix a point $s_0 \in S$ and the corresponding fiber $X=X_{s_0}$ together with the map $f=f_{s_0}$. The monomorphism of sheaves 
$f_*: T_X \to f^*T_{Y}$ induces a monomorphism
$$
f_*: A^{0,1}(T_X) \to A^{0,1}(f^*T_Y)
$$
of global $C^{\infty}$-forms of type $(0,1)$. By means of this map $f_*$, we can view in particular the harmonic $(0,1)$-forms on 
$X$ with values in $T_X$ as $(0,1)$-forms with values in $f^*T_Y$. (Note that the images in $A^{0,1}(f^*T_Y)$ are in general no longer harmonic with respect to a hermitian metric on $Y$.) We denote the space of harmonic forms by $\Harm^{0,1}(X,T_X)$.

\begin{proposition}
\label{HXY}
$$
D'_{X/Y} \cong \frac{\{ \chi \in A^{0,0}(f^*T_Y) \, : \, \dbar \chi \in f_*(\Harm^{0,1}(X,T_X)) \}}{\{f_*(\zeta) \; : \;   \zeta \in H^0(X,T_X) \}} =: H_{X/Y}
$$
\end{proposition}
\begin{proof}
Since $f$ is non-degenerate, we can simplify $D'_{X/Y}$ to  
$$
D'_{X/Y} \cong \frac{\{ \xi \in A^{0,0}(X,f^*T_Y) \; : \; \dbar \xi \in f_*(A^{0,1}(T_X)) \}}{ \{  f_*(\zeta) \; : \; \zeta \in A^{0,0}(T_X) \} } =: D''_{X/Y}
$$
Then there is a natural map from $H_{X/Y}$ to $D''_{X/Y}$. We have to show that this map is indeed bijective.
Let $\xi \in A^{0,0}(f^*T_Y)$ such that $\dbar \xi = f_* \vartheta$  for a $\vartheta \in A^{0,1}(T_X)$. There exists a $\zeta \in A^{0,0}(T_X)$ such that $\vartheta'=\vartheta + \dbar \zeta$ is harmonic. Then $\chi = \xi + f_* \zeta$ lies in the same class as $\xi$ in $D''_{X/Y}$ and we have  $\dbar \chi = f_* \vartheta'$. If $\chi, \chi' \in A^{0,0}(f^*T_Y)$ are given such that $\dbar \chi = f_* \vartheta$ and $\dbar \chi' = f_* \vartheta'$ for harmonic forms $\vartheta, \vartheta' \in A^{0,1}(T_X)$ and $\chi' = \chi + f_* \zeta$, then it follows by
$f_* \vartheta' = \dbar \chi' = \dbar (\chi + f_* \zeta) = f_* (\vartheta + \dbar \zeta)$ and the injectivity of $f_*$ that $\vartheta' = \vartheta + \dbar \zeta$. Thus the two harmonic forms $\vartheta$ and $\vartheta'$ are Dolbeault equivalent,  hence coincide. It follows $\dbar \zeta =0$, i.e. $\zeta$ is a global holomorphic vector field on $X$. Therefore, the vector fileds 
$\chi$ and $\chi'$ determine the same class  in $H_{X/Y}$.  
\end{proof}

\begin{remark}
We consider the case $H^0(X,T_X)=0$. We get
$$
D_{X/Y} \cong \{ \chi \in A^{0,0}(f^*T_Y) \, : \, \dbar \chi \in f_*(\Harm^{0,1}(X,T_X)) \} = H_{X/Y}
$$
and we write the exact sequence $\ref{seq1a}$ in the form
$$
0 \to H^0(X,f^*T_Y) \xrightarrow{\la} H_{X/Y} \xrightarrow{\lb} f_*(\Harm^{0,1}(X,T_X)) \xrightarrow{\lc} H^1(X,f^*T_Y) 
$$
Now all maps have a very simple description: $\la$ can be read as an inclusion. The map $\lb$ is simply $\dbar$ and $\lc$ assigns to a form $f_*(\vartheta) \in A^{0,1}(X,f^*T_Y)$ its Dolbeault class in $H^1(X,f^*T_Y)$. The exactness comes immediately. 
\end{remark}
Form now on let also $Y$ be equipped with a hermitian metric $h$. The pullback  $(f^*T_Y,f^*h)$ is then a hermitian bundle on 
$X$. We denote the adjoint operator of $\dbar$ with respect to $f^*h$ on the space $A^{0,1}(X,f^*T_Y)$ by $\dbar^*_h$. On the space $A^{0,0}(X,f^*T_Y)$ and $A^{0,1}(X,f^*T_Y)$, we denote the harmonic projection and the Green operator by $H$ and $G$ respectively. It is well-kown that $G$ commutes with $\dbar$ and $\dbar^*_h$.
\begin{proposition}
\label{Pos}
Let $H^0(X,T_X)=0=H^1(X,f^*T_Y)$. We interpret the exact sequence
$$
0 \to H^0(X,f^*T_Y) \to H^0(X,N_f) \to H^1(X,T_X) \to 0
$$
as
$$
0 \to H^0(X,f^*T_Y) \xrightarrow{\iota} \{\chi \in A^{0,0}(f^*T_Y) \, : \, \dbar \chi \in f_*(\Harm^{0,1}(X,T_X)) \} \xrightarrow{\dbar} f_*(\Harm^{0,1}(X,T_X)) \to 0.
$$
Then we have a splitting by means of the following maps:
$$
	 0 \xleftarrow{}  H^0(X,f^*T_Y) \xleftarrow{H}  H_{X/Y} \xleftarrow{G \dbar^*_h}  f_*(\Harm^{0,1}(X,T_X)) \xleftarrow{} 0 
$$
\end{proposition}
\begin{proof}
We have the identity $\id = H + G \laplacedbar$, where $\laplacedbar = \dbar^*_h \dbar$ for elements in $A^{0,0}(X,f^*T_Y)$.
\end{proof} 
\begin{remark}
Since $H^1(X,f^*T_Y)=0$, there are no harmonic $(0,1)$-forms with values in $f^*T_Y$. Hence, the Laplace operator 
$\laplacedbar= \dbar^*_h \dbar$ is invertible on the space $A^{0,1}(X,f^*T_Y)$ and the inverse operator is the Green operator on 
$A^{0,1}(X,f^*T_Y)$. Therefore, we have in particular 
$u=G \dbar \dbar^*_h u = \dbar G \dbar^*_h u$ for all $u \in f_*(\Harm^{0,1}(X,T_X))$.
\end{remark}    
\begin{proposition}
\label{Neg} 
Now let $H^0(X,T_X)=0=H^0(X,f^*T_Y)$. We write the exact sequence
$$
0 \to H^0(X,N_f) \to H^1(X,T_X) \to H^1(X,f^*T_Y) \to 0
$$
as
$$
0 \to \{\chi \in A^{0,0}(f^*T_Y) \, : \, \dbar \chi \in F(\Harm^{0,1}(X,T_X)) \} \xrightarrow{\dbar} f_*(\Harm^{0,1}(X,T_X)) \xrightarrow{H} \Harm^{0,1}(X,f^*T_Y)  \to 0,
$$
Then we have a splitting by means of the following maps:
$$
	 0 \xleftarrow{}  H_{X/Y} \xleftarrow{\dbar^*_h G}  f_*(\Harm^{0,1}(X,T_X)) \xleftarrow{\iota} \Harm^{0,1}(X,f^*T_Y) \xleftarrow{} 0 
$$
\end{proposition}
\begin{proof}
We have the identity $\id = H + \laplacedbar G$, where $\laplacedbar = \dbar \dbar^*_h$ for elements in $A^{0,1}(f^*T_Y)$.
\end{proof} 
\begin{remark}
Proposotion $\ref{Pos}$ asserts that one can assign on the infinitesimal level to a deformation of $f: X \to Y$, i.e. an element of $H_{X/Y}$, an element of $H^0(X,f^*T_Y)$, i.e. a deformation where the complex structure remains unchanged. The second proposition $\ref{Neg}$ means that one can assign to an infinitesimal deformation of $X$ an infinitesimal deformation of $f:X \to Y$.
\end{remark}
\begin{remark}
It is not difficult to generalize the results of this section to the case of general deformations of holomorphic maps which do not fix the target space $Y$. 
\end{remark}

\subsection{Movement of branching points}
We end this section with a rather concrete infinitesimal consideration of the Hurwitz space. For this, 
we consider the universal family $\X \to \PG \times \Hbn$. As we have seen in the introduction, the complex structure on $\Hbn$ is given by the finite unbranched topological covering
$$
br: \Hbn \to \operatorname{Sym}^b \PG \setminus \Delta =: \PG_b \setminus \Delta 
$$ 
This map assigns to a branched covering its set of branching points on $\PG$, which we read as a divisor on $\PG$. We write for short $H=\Hbn$ and consider an arbitrary point $s_0 \in H$ as well as the corresponding covering 
$\beta=\beta_{s_0}: X=X_{s_0} \to \PG$. We write $B=br(s_0)$ for the branching divisor. Then the differential
$$
br_*: T_{s_0}(H) \to T_D(\PG_b)
$$ 
is an isomorphism. There is an intrinsic isomorphism (see \cite[p. 160]{ACGH85})
$$
T_D(\PG_b) \cong H^0(\PG,\Oh_B(B))
$$
The space $H^0(\PG,\Oh_B(B))$ is known as the space of infinitesimal deformations of the effective divisor $B$ on $\PG$, see
\cite[p. 94]{HM98}. We take the geometric point of view and interpret a tangent vector at $s_0 \in H$ as an equivalence class of smooth curves in $H$ through the point $s_0$ under the equivalence of first order approximation. These corresponds via the map $br$ to curves in the space of branching points or differently speaking: The infinitesimal movement of  the $b$ branching points in $\PG_b$ corresponds to an infinitesimal deformation of the covering $\beta: X \to \PG$. But we have seen that these infinitesimal deformations are in turn described by the space
$$
T_{s_0}(H) \cong H^0(X,N_{\beta}) \cong H_{X/Y},
$$
which appears in the tangent sequence
$$
0 \to H^0(X,\beta^*T_{\PG}) \to H_{X/ \PG} \to H^1(X,T_X) \to 0
$$
An element in $H^0(X,N_f)$ is represented by a differentiable vector field on $X$, which is meromorphic in a neighborhood of the ramification points with simple poles there (compare also the presentation in \cite[pp. 126-128]{NR11}). If $z^1,\ldots,z^b$ are local coordinates centered around the branching points $x_1,\ldots,x_b$, this vector field has the local expression
$$
\sum_{\nu =-1}^{\infty}{a_v^jz_j^v}
$$
around the point $x_j$. Let $c_j=a_{-1}^j$ be the residuum of this vector field at the point $x_j$ with respect to the coordinate 
$z_j$. By means of this coordinates we can identify $H^0(X,\Oh_R)$ with $\C^b$. The vector field then represents the value 
$(c_1,c_2,\ldots,c_b)$. The subspace $H^0(X,\beta^*T_{\PG}) \subset H_{X/ \PG}$ stands for the infinitesimal deformations of $\beta: X \to \PG$, where the complex structure of $X$ remains (infinitesimal) unchanged. This means that the complex structure remains infinitesimally unchanged along directions of vectors, which are residuals of holomorphic vector fields with simple poles in the ramification points. The residuum of a meromorphic vector field depends on the chosen coordinates. However, the property $c=0$ and $c\neq0$ respectively does not!  Now we can ask the question which movements of branching points keep the structure on $X$ fixed or not. One result in this direction is the following statement:
\begin{proposition}
If one moves less than $2g-2$ branching points on $\PG$, the complex structure on $X$ changes. 
\end{proposition}
\begin{proof}
Let
$$
R=\sum_{i=1}^b{p_i}
$$ 
be the ramification divisor of $\beta$ on $X$ and 
$$B=\sum_{i=1}^b{q_i}
$$
the branching divisor on $\PG$, where $\beta(p_i)=q_i$. Since we have a simple covering $\beta: X \to \PG$, there is a one-to-one correspondence between the ramification points on $X$ and the branching points on $\PG$. An infinitesimal movement of the points $q_1,\ldots, q_b$ is an element of $H^0(\PG,\Oh_B(B))$. In local coordinates $z_i$ centered around the points $q_i$, such a section can be written as
$$
\eta = \sum_{i=1}^b \eta_i
$$ 
where
$$
\eta_i = a_i \frac{1}{z_i}
$$
If the $k$-th point is not moved, we have (in all such coordinates) $a_k=0$. By means of the identifications
$$
H^0(\PG,\Oh_B(B)) \cong T_B(\PG_b) \cong T_{s_0}(H) \cong H^0(X,N_{\beta}) \cong H^0(X,\Oh_R) 
$$ 
such a section gives a differentiable vector field on $X$, which is meromorphic in a neighborhood of the ramification points.
In this situation we have a pole (that means the "residuum" is different from zero in all centered coordinates) if and only if the corresponding branching point moves. Now the question is the following: Is there a meromorphic vector field on $X$, which has simple poles in only $r<b$ of the $b$ branching points? For $r< 2g-2$, the answer is negative, since the line bundle $T_X(R')$
has negative degree for a divisor $R'<R$ of degree $r$ and hence has no non-trivial holomorphic sections.     
\end{proof}

\begin{remark}
The statement of the proposition can already be found in the note \cite{Fr12}. There, using methods from Teichm\"uller theory, the statement is  proved by assigning a Beltrami differential to each movement of branching points. This leads to a Riemann-Roch discussion of the space $H^0(X,K_X(-B))$, the Serre dual of the space $H^1(X,T_X(B))$.
\end{remark}

\section{The Weil-Petersson metric}
\subsection{Families of coverings}
We apply the deformation theory developed in the last section to the case of branched coverings of compact Riemann surfaces. We fix a compact Riemann surface $Y$ of arbitrary genus. We use the notation and definitions from \cite{ABS15}.
\begin{definition}
A \emph{holomorphic family of coverings of Riemann surfaces} is a holomorphic map $\Phi=(\lb,f): \X \to Y \times S$, where the maps $\lb_s: X_s \to Y$ are non-constant holomorphic maps of a compact Riemann surfaces $X_s$ to $Y$.
\end{definition}

Now let $(\lb,f): \X \to Y \times S$ be a holomorphic family of coverings with $g(X_s)>1$, i.e. $X_s$ is hyperbolic. We fix a point $s_0 \in S$ together with the fiber $X=X_{s_0}$ and the corresponding map $\lb_0=\lb_{s_0}: X \to Y$. Let $z$ be a local holomorphic coordinate on the fibers and $s=(s^1,\ldots,s^r)$ be local holomorphic coordinates on $S$, which we use as coordinates on $\X$ such that $f(z,s)=s$.  Then every fiber $X_s$ carries a unique hyperbolic metric
$$
\omega_{X_s} = \sqrt{-1}g(z,s) dz \wedge d \zbar
$$
of constant Ricci curvature $-1$, which is $C^{\infty}$ and also depends $C^{\infty}$-differentiable on the parameter $s$. It therefore holds
\begin{equation}
\label{Ricci}
\frac{\dl^2}{\dl z \dl \zbar} \log g(z,s) = g(z,s).
\end{equation}
The K\"ahler forms on the fibers of $f$ yield a hermitian metric $g^{-1}(z,s)$ on the relative canonical bundle $K_{\X/S}$. We denote its curvature form by
$$
\omega_{\X} = \sqrt{-1} \dl \dbar \log g(z,s).
$$
Because of $\ref{Ricci}$ we have
$$
\omega_{\X}|_{X_s}=\omega_{X_s}.
$$
Let
$$
\rho_s: T_sS \to H^1(X_s,T_{X_s})
$$ 
be the Kodaira-Spencer map of the deformation $f: \X \to S$ at $s \in S$. Let $\dl/\dl s=\dl_s$ be a tangent vector in $T_sS$. Harmonic representatives of $\rho(\dl_s)$ with respect to the hyperbolic metrics on $X_s$ are harmonic Beltrami differentials, which we denote by
$\mu_s=\mu_{s \zbar}^z \dl_s d \zbar$. These objects can be obtained by \emph{horizontal lifts}  of $\dl_s$ (see \cite{Sch93}), which are also \emph{canonical lifts} in the sense of Siu (\cite{Siu86}). This lift can be computed as
\begin{eqnarray}
\label{HorLift}
v_s=\dl_s + a_s^z \dl_s \quad \mbox{where} \quad a_s^z=-g^{\zbar z} g_{s \zbar}
\end{eqnarray}
which indeed gives a lift of $\dl_s$ perpendicular to the fibers with respect to $\omega_{\X}$.  
In this notation $g_{s \zbar}$ is the component of $\omega_{\X}$ in the direction of $z$ and $s$. 
The harmonic Beltrami differential is given by
$$
\mu_s = (\dbar v_s)|_{X_s} = \dl_{\zbar}(a_s^z)\dl_z d \zbar.
$$
Now we consider the characterisitic map
$$
\tau_s : T_sS \to H^0(X_s,N_{\beta_s}).
$$
Since we have chosen a K\"ahler metric $\omega_{X_s}$ on $X_s$ and there are no non-trivial holomorphic vector fields on 
$X_s$, we can identify $H^0(X_s,N_{\beta_s})$ with 
$$
H_{X_s/Y} = \{ \chi \in A^{0,0}(\beta_s^*T_Y) \, : \, \dbar \chi \in \beta_{s, *}(\Harm^{0,1}(X,T_X)) \}
$$ 
by the results of the last section. One obtains a representative of $\tau_s(\dl_s)$ in $H_{X_s/Y}$ also by using the horizontal lift  $v_s$: 
$$
u_s := \beta_*{v_s} = \beta_*(\dl_s + a_s^z \dl_z) = \frac{\dl \beta}{\dl s} \frac{\dl}{\dl w} + a_s^z \frac{\dl \beta}{\dl z} \frac{\dl}{\dl w} = (\xi^w_s + a_s^z \zeta^w_z) \frac{\dl}{\dl w}, 
$$
where we introduced the notation 
$$
\xi^w_z = \dl \beta / \dl s, \quad \zeta^w_z = \dl \beta / \dl z.
$$ 
This gives a differentiable vector field $X_s$ with values in 
$\beta_s^*T_Y$ such that
$$
\dbar(u_s) = \mu_s \zeta^w_z \dl_w d \zbar \in \beta_{s, \star}(\Harm^{0,1}(X_s,T_{X_s})). 
$$
By Proposition $\ref{diffbarehochhebung}$ and the proof of Proposition $\ref{HXY}$ we have $\tau_s(\dl_s)=u_s \in H_{X_s/Y}$. We call the elements of $H_{X_s/Y}$ (by abuse of notation) \emph{generalized harmonic representatives} of the characteristic map. We also state this as a result:
\begin{proposition}
The pushforward of horizontal lifts of tangent vectors give generalized harmonic representatives of the characteristic map.
\end{proposition}
\begin{remark}
The vector fields $u_s$ have already been constructed in the work \cite{ABS15} without being aware of their deformation theoretic meaning. These ad hoc defined objects were used for defining the Weil-Petersson metric, see below.
\end{remark}

\subsection{Definition of the metric}
We keep the notation from the previous subsection. Using the harmonic representatives of the characteristic map, we are able to introduce a metric on any smooth base $S$ of an effectively parametrized family $(\beta,f): \X \to Y \times S$ of coverings of Riemann surfaces. Let $s=(s^1, \cdots, s^r)$ be again local holomorphic coordinates on $S$, which we use together with a coordinate $z$ on the fiber as coordinates on $\X$ such that $f(z,s)=s$. Moreover, let 
$w$ be a local holomorphic coordinate on $Y$ so that $\beta(z,s)=w.$ The K\"ahler metric on $Y$ of constant Ricci curvature (+1, 0 or -1 depending on $g(Y)$) reads as 
$$\omega_Y=\sqrt{-1} \hw \; dw \wedge d \wbar.
$$
For the pullback of this metric we get
$$
\beta^* \omega_Y = \sqrt{-1}\hw(\beta(z,s))\left( 
\z \zb \, dz \wedge d \zbar + \z \xi^{\wbar}_{\jbar} dz \wedge d s^{\jbar} + \xi^w_i \zb ds^i \wedge d \zbar + \xi^w_i \xi^{\wbar}_{\jbar} ds^i \wedge ds^{\jbar}
\right)
$$ 
Here and in the following we are using Einstein's convention of summation. The explicit expression for  $\omega_{\X}=\sqrt{-1}\dl \dbar \log g(z,s)$ gives
$$
\omega_{\X}=\sqrt{-1}(g_{z \zbar} dz \wedge d \zbar + g_{z\jbar} dz \wedge ds^{\jbar} + g_{i \zbar} ds^i \wedge d \zbar + g_{i \jbar} ds^i \wedge ds^{\jbar}).
$$
For a tangent vector $\dl_i=\dl / \dl s^i, 1 \leq i \leq r$, we set $u_i:= u_{\dl/\dl s^i}$. For any $s \in S$ the space
$H_{X_s/Y} \subset A^{0,0}(X_s,\beta_s^*T_Y)$ carries a natural scalar product. Thus we define 
$$
G^{WP}_{1,i \, \jbar}(s):=\int_{X_s}{u_i \cdot \ovl{u_j} \, g dA} 
= \int_{X_s}{(\xi^w_i + a_i^z \zeta^w_z)(\xi^{\wbar}_{\jbar} + a_{\jbar}^{\zbar} \zeta^{\wbar}_{\jbar}) \, \hw(\beta(z,s)) \, \sqrt{-1} \, g_{z \zbar}(z,s) \, dz \wedge d \zbar}.
$$
Here we were writing $gdA=\sqrt{-1}\gz(z,s)dz \wedge d\zbar$ for the area element with respect to the hyperbolic metric on the fiber $X_s$. 
As we will see soon, already this product gives a hermitian metric on the base $S$ if the family is effectively parametrized as a family of maps. But for obtaining a K\"ahler metric, we need a second term, see \cite{ABS15}. We set
$$
G^{WP}_{0,i \, \jbar}:=\int_{X_s}{\varphi_{i \jbar} \, \beta_s^*\omega_Y} = \int_{X_s}{\varphi_{i \, \jbar}(z,s) \hw(\beta(z,s)) \z(z,s) \zb(z,s) \sqrt{-1} dz \wedge d \zbar}
$$
Here we used the function
$$
\varphi_{i \, \jbar} := \langle v_i, v_j \rangle_{\omega_{\X}} = g_{i \jbar} - g_{i \zbar} g_{z \jbar} g^{\zbar z},
$$
which is the inner product of the horizontal lifts in the direction of $i$ and $j$ with respect to the form $\omega_{\X}$. One should compare the expression for $G^{WP}_{0}$ with the easier expression
$$
\int_{X_s}{\varphi_{i \jbar} \; gdA} = \int_{X_s}{\varphi_{i \jbar}\; \omega_{X_s}}
$$
of the Weil-Petersson metric for the family of compact Riemann surfaces $f: \X \to S$ (see \cite{LSY04}).
We define the \emph{Weil-Petersson inner product} for two tangent vectors $\dl/\dl s^i$ and $\dl/ \dl s^j$ at the point $s$ by
$$
\langle \dl/\dl s^i, \dl/ \dl ^js \rangle_{WP} := G^{WP}_{i \, \jbar}(s) := G^{WP}_{0, i \, \jbar}(s) +  G^{WP}_{1, i \, \jbar}(s).
$$
\subsection{Positivity of the pairing}
\begin{proposition}\cite{ABS15}
The Weil-Petersson product is positive definite, if the family
$$
(\beta,f): \X \to Y \times S
$$
is effectively parametrized. 
\end{proposition} 
\begin{proof}
We present a more conceptional proof than in \cite{ABS15}. 
We proof the property at a point $s \in S$ and choose a tangent vector $\dl_s \in T_sS$ such that $G_{s\sbar}^{WP}(s)$ vanishes. First, we consider
$$
G_{0,s\sbar}^{WP}(s)= \int_{X_s}{\varphi_{s\sbar}(z,s)\hw(\beta(z,s))\z(z,s)\zb(z,s)\sqrt{-1}dz\wedge d\zbar}.
$$
From the equality
$$
(\Box_{\omega_s} + 1)\varphi_{s \sbar} = ||\mu_s||^2
$$
and the fact that the operator $\Box_{\omega_s}+1$ is strictly positive (see \cite{Sch12}), it follows that $\varphi_{s\sbar}$ is non-negative. Since also the second term $G_{1,s\sbar}^{WP}(s)$ is non-negative, we have
$$
G_{0,s\sbar}^{WP}(s)=0=G_{1,s\sbar}^{WP}(s).
$$ 
Now it follows from
$$
G_{1,s\sbar}^{WP}(s)=\int_{X_s}{||u_s||^2\,gdA}=0
$$
that $u_s=0$ as an element of $H_{X_s/Y}\subset A^{0,0}(X_s,\beta_s^*T_Y)$. Hence, $\tau_s(\dl_s)=0$ and since $\tau_s$ is injective for a effectively parametrized family, we finally get $\dl_s=0$. 
\end{proof}
\begin{remark}
The vanishing of $u_s$ means for the local expression $\xi^w_s + a_s^z\z=0$, i.e. $a_s^z=-\xi^w_s/\z.$ Since $a_s^z$ is everywhere differentiable and hence does not have any poles, $a_s^z$ must therefore be holomorphic. Then the family of complex structures $f: \X \to S$ is infinitesimal trivial at the point $s$ in the direction of $\dl_s$ and $u_s$ is an element of 
$H^0(X_s,\beta_s^*T_Y)$. But this element is zero, so the family of coverings $\beta: \X \to Y$ is also infinitesimal trivial at $s$ in the direction of $\dl_s$. 
\end{remark}
\subsection{The choice of coordinates}
We consider the case of families of coverings $(\beta,f): \X \to Y \times S$ of compact hyperbolic Riemann surfaces $X_s$ over a one-dimensional base $0\in S \subset \C$. We assume that the horizontal lift
$$
v_s = \dl_s + a_s^z\dl_z
$$
of the coordinate vector field $\dl_s$ is a holomorphic vector field on $\X$. By integrating this vector field (after a shrinking of 
$S$),
we obtain a trivialization
$$
\Phi: X_0 \times S \to \X, \; (\tz,s) \mapsto (z(\tz,s),s)
$$
such that
$$
\frac{\dl z}{\dl s}(\tz,s) = a_s^z(z(\tz,s),s).
$$
We now compute the family of metric tensors in this new trivialising coordinates $\tz,s$:
$$
g(z,s)=g(z(\tz,s),s)=: \tg(\tz,s).
$$
We obtain for the derivatives $\log \tg$ 
$$
\frac{\dl \log \tg}{\dl s} = \frac{\dl \log g}{\dl z} \frac{\dl z}{\dl s} + \frac{\dl \log g}{\dl s}
$$
und further
\begin{eqnarray*}
\tg_{s \ovl{\tz}}=
\frac{\dl^2 \log \tg}{\dl \ovl{\tz} \dl s} &=& \frac{\dl^2 \log g}{\dl \zbar \dl z} \frac{\dl \zbar}{\dl \ovl{\tz}} \frac{\dl z}{\dl s} + \frac{\dl^2 \log g}{\dl \zbar \dl s} \frac{\dl \zbar}{\dl \ovl{\tz}} \\
&=& (g_{z \zbar}a_s^z + g_{s \zbar}) \dl \zbar/\dl \ovl{\tz}\\
&=& 0,
\end{eqnarray*}
since
$$
a_s^z=-g^{\zbar z}g_{s \zbar}.
$$
Hence, we have in the new coordinates $a_s^{\tz}=0$, i.e. $v_s=\dl_s$. Now we compute $u_s$ in this trivialising coordinates. We set
$$
\beta(z,s)=\beta(z(\tz,s),s)=: \tbeta(\tz,s)=:\tw(\tz,s).
$$
Then
\begin{eqnarray*}
\tilde{\xi}^{\tw}_s &=& \frac{\dl \tbeta}{\dl s} = \frac{\dl \beta}{\dl s} + \frac{\dl \beta}{\dl z}\frac{\dl z}{\dl s} \\
&=& \xi^w_s + \zeta^w_z a_s^z
\end{eqnarray*}
and $u_s=\tilde{\xi}^{\tw}_s \dl_{\tw}$. We observe
\begin{proposition}
\label{LokTriv}
A family $(\beta,f): \X \to Y \times S$ is locally trivial as a family of coverings at a point $s_0 \in S$, that means the restriction of the family to a neighborhood of $s_0 \in S' \subset S$ is isomorphic to $(\beta_{s_0} \times \operatorname{pr}_2): X_0 \times S' \to Y \times S'$ if and only if $u_s(z,s)=0$ entirely on $\X|_{S'}$.
\end{proposition}
\subsection{Preparations and useful identities}
In this subsection, we recall briefly the calculus of covariant derivatives, which we apply to global $C^{\infty}$-sections of the hermitian bundle $(\beta^*T_Y,\beta^*h)$. Furthermore, we collect some useful formulas for the further computations, which already appear in the computation of the curvature of the Weil-Petersson metric on the Teichmüller space.   

By keeping the notation from the previous subsections, we use the symbol $|$ for ordinary and $;$ for covariant derivatives. We set $\dl_z=\dl/\dl z, \dl_{\zbar}=\dl/\dl \zbar$ and $\dl_k=\dl/\dl s^k, \dl_{\lbar}=\dl/\dl s^{\lbar}$ for coordinate directions $1\leq k,l\leq r$ on the $r$-dimensional base $S$.  Let $u=u^w(z,s)\dl_w$ and $v=v^w(z,s)\dl_w$ be vector fields along the fibers of $\X \to S$, i.e. $u_s:=u(z,s)$ and $v_s=v(z,s)$ are differentiable families of vector fields with values in $\beta_s^*T_Y$. Then 
\begin{eqnarray*}
\dl_z(u,v)&=& \dl_z (u \cdot \vbar) \\ 
&=&  \dl_z(u^w v^{\wbar} \hw)\\
&=& u^w_{|z}v^{\wbar}\hw + u^w v^{\wbar}_{|z}\hw + u^w v^{\wbar} h_{w\wbar|z} \\ 
&=& u^w_{|z}v^{\wbar}\hw + u^w v^{\wbar}_{|z}\hw + u^w v^{\wbar} \Gamma_w \z \hw \\
&=& u^w_{;z} \vbar \hw + u^w \vbar_{,z} \hw \\
&=& \nabla_z(u) \cdot \vbar + u \cdot \nabla_z(\vbar) \\
&=& (\nabla_z(u),v) + (u,\dl_{\zbar}(v)),
\end{eqnarray*}
where we introduced the covariant derivatives
\[
\nabla_z(u) = (u^w_{|z} + \Gamma_w \z u^w)\dl_w
\]
and
\[
\nabla_z(\vbar) = (v^{\wbar}_{|z})\dl_w.
\]
Here we set
$$
\Gamma_w=\Gamma_w(\beta(z,s))=(h(\beta(z,s)))^{-1}(\dl_wh)(\beta(z,s)).
$$
Analogously we have for $1\leq k \leq r$
\[
\dl_k(u,v) = (\nabla_k(u),v) + (u,\dl_{\kbar}(v)),
\]
where
\[
\nabla_k(u) = (u^w_{|k} + \Gamma_w \xi^w_k u^w)\dl_w.
\]

\begin{lemma}
$\nabla_z(\hw) = \nabla_k(\hw) = 0.$
\end{lemma}
\begin{proof}
We compute
\[
h_{w\wbar;z} = h_{w\wbar|z} - \Gamma_w \z \hw = \Gamma_w \z \hw - \Gamma_w \z \hw = 0.
\]
Analogously
\[
h_{w\wbar;k} = h_{w\wbar|k} - \Gamma_w \xi^w_k \hw = \Gamma_w \xi^w_k \hw - \Gamma_w \xi^w_k \hw = 0.
\]
\end{proof}

Furthermore, we need some useful formulas for the computations in the next sections, which already appear in the computation of the curvature of the Weil-Petersson metric on the Teichm\"uller space:

\begin{lemma}
\label{Liu}
The following equations hold (see \cite[proof of Lemma 3.3]{LSY04}):
\begin{eqnarray}
\dl_z a_i  & = & -\Gamma_z a_i - \dl_i \log g \\ 
\dl_{\zbar} a_{\jbar}& = & -\Gamma_{\zbar} a_{\jbar} - \dl_{\jbar} \log g \\
\dl_{\lbar} a_i & = & -A_ia_{\lbar} - g^{-1} \dl_{\zbar} \varphi_{i \lbar}\\
\dl_k a_{\jbar} & = & -A_{\jbar} a_k - g^{-1} \dl_z \varphi_{k \jbar}
\end{eqnarray}
Here we have $\Gamma_z=g^{-1}\dl_zg$.
\end{lemma}
Moreover, we have the following results (see \cite{Sch93,Sch12}):
\begin{lemma}
\label{diffInt}
We write $L_{k}$ for the Lie derivative with respect to the vector field $v_k$. Then
$$
\dl_k \int_{X_s}{\eta} = \int_{X_s}{L_k(\eta)} \quad \mbox{and} \quad \dl_{\lbar} \int_{X_s}{\eta} = \int_{X_s}{L_{\lbar}(\eta)}
$$
for any smooth $(1,1)$-form $\eta$ on $\X$.
\end{lemma}
\begin{lemma}
\label{LieDer}
$L_{k}(\gz\; dz\wedge d\zbar)=0$
\end{lemma}
\begin{lemma}
$\varphi_{i\jbar} = g_{i\jbar} - g a_i a_{\jbar}$
\end{lemma}

\begin{lemma}
\label{ellipGl}
$(\Box + 1)\varphi_{i\jbar} = A_i\cdot A_{\jbar},$
\end{lemma}
The form $A_{i\zbar}^z(z,s)\dl_zd\zbar=\dbar v_i|_{X_s}$ is the harmonic representative of the Kodaira-Spencer class
$\rho(\dl_i)$. 

The fact that the operator $(\Box + 1)$ is invertible gives 

\begin{corollary}
\label{InfTriv}
A deformation $f: \X \to S$ is infinitesimal trivial at a point $s_0 \in S$ in the direction of $1\leq i \leq r$ or $1 \leq j \leq r$ if and only if $\varphi_{i\jbar}(z,s_0)=0$ for all $z \in X_{s_0}$.
\end{corollary}
\begin{remark}
\label{BemHolLifts}
For simplicity we consider a one-dimensional base $S$ with a local coordinate $s$. It follows from equation $(4.5)$ of Lemma 
$\ref{Liu}$ and the preceding corollary: If $f: \X \to S$ is infinitesimal trivial at $s_0 \in S$ , then $\dl_{\zbar}a_s = 0 = \dl_{\ovl{s}}a_s$. Thus, the horizontal lift
$
v_s = \dl_s + a_s^z \dl_z
$
is holomorphic with respect to $z$ and $s$ if the family $f:\X \to S$ is infinitesimal trivial entirely on $S$.  By integrating this holomorphic horizontal lift, we obtain a local trivialization of the family.
\end{remark}
The vector field $v_k$ is a horizontal lift of $\dl_k$ with respect to the form $\omega_{\X}$. Now we ask for a horizontal lift with respect to $\beta^*\omega_Y$. Since this form has zeros, we obtain a vector field with poles:
\begin{proposition}
The horizontal lift $\tilde{v}_k$ of $\dl_k$ with respect to the form $\beta^*\omega_Y$ is given by
$$
\tilde{v}_k=\dl_k - (\xi^w_k/\z)\dl_z=\dl_k - \xi^w_k\dl_w.
$$
\end{proposition}
\begin{proof}
The ansatz
$$
\tilde{v}_k=\dl_k+b_k^z\dl_z
$$ and the condition 
$$
\langle \dl_k + b_k^z\dl_z,\dl_z \rangle_{\beta^*\omega_Y}=0
$$
lead to $b_k^z=-\xi^w_k/\z$.
\end{proof}
Analougsly to \ref{LieDer} we get
\begin{lemma}
$L_{\tilde{v}_k}(\hw \z\zb \; dz \wedge d\zbar)=0$.
\end{lemma}
\begin{proof}
We write $\hz=\hw\z\zb$ und $|$ for an ordinary derivative. Then
\begin{eqnarray*}
L_{\tilde{v}_k}(\hz)_{z\zbar}&=&[\dl_k-(\xi^w_k/\z),\hz]_{z\zbar}\\
&=& h_{z\zbar|k} -(\xi^w_k/\z)h_{z\zbar|z} - (\xi^w_k/\z)_{|z}\hz \\
&=& \xi^w_k\hw\Gamma_w\z\zb + \hw\zeta^w_{z|k}\zb\\
&-&  \xi^w_k\hw\Gamma_w\z\zb - (\xi^w_k/\z)\hw\zeta^w_{z|z}\zb\\
&+&  \xi^w_{k|z}\hw\zb + (\xi^w_k/\z)\hw\zb\zeta^w_{z|z}\\
&=&0.
\end{eqnarray*}
\end{proof}
\begin{remark}
The generalized harmonic representative $u_k$ is the difference of the classical horizontal lift $v_s$ and the horizontal lift 
$\tilde{v}_k$.
\end{remark}

\subsection{K\"ahler property}
Very often, the K\"ahler property of the Weil-Petersson metric follows from a fiber integral formula, because the exterior differential commutes with the fiber integral. In this subsection, we prove the K\"ahler symmetry by a direct computation, which gives us the insight that the single expression $G_1^{WP}$ alone does not yield a K\"ahler metric on the Hurwitz space. 

The K\"ahler property means that $d\omega^{WP}=0$. This is equivalent to the K\"ahler symmetry
$$
\dl_kG_{i\jbar}^{WP}(s) = \dl_iG_{k\jbar}^{WP}(s)
$$
for all $1\leq i,j,k\leq r=\dim S$ and $s \in S$. We compute by using the Lie derivative $L_k$ with respect to the vector field $v_k$ and Lemma \ref{diffInt}:
\begin{eqnarray*}
		\partial_k \G_{i\ovl{j}}(s)
&=& \partial_k G^{WP}_{0,i\ovl{j}}(s) + \partial_k G^{WP}_{0,i\ovl{j}}(s)\\
&=& \partial_k \int_{X_s}{\varphi_{i\jbar}\beta_s^*(\wY)} + \partial_k \int_{X_s}{u_i\cdot u_{\jbar}g dA}\\
&=& \int_{X_s}{L_k(\varphi_{i\jbar})\beta_s^*(\wY)} + \int_{X_s}{\varphi_{i\jbar}L_k(\beta_s^*(\wY))}\\
&+& \int_{X_s}{L_k(u_i \cdot u_{\jbar})g dA} + \int_{X_s}{u_i\cdot u_{\jbar} L_k(g dA)}\\
&=& \int_{X_s}{L_k(\varphi_{i\jbar})\beta_s^*(\wY)} + \int_{X_s}{\varphi_{i\jbar}L_k(\beta_s^*(\wY))}\\
&+& \int_{X_s}{D_k(u_i) \cdot u_{\jbar}g dA} + \int_{X_s}{u_i\cdot D_k(u_{\jbar})g dA},\\
\label{ksym}
\end{eqnarray*}
where we introduced the notation
\[
D_k(u_i) := (\nabla_k + a^z_k \nabla_z)(u_i)
\]
for the covariant derivative in the direction of the horizontal lift $v_k$.  We also made use of Lemma \ref{LieDer}.
Now $L_k(\varphi_{i\jbar}) = v_k(\varphi_{i\jbar}) = v_i(\varphi_{k\jbar})=L_i(\varphi_{k\jbar})$ (see \cite[Lemma 3.2]{LSY04}).
Furthermore, also the third summand is symmetric in $i$ and $k$:
\begin{eqnarray*}
D_k(u_i) &=& (\nabla_k + a_k^z\nabla_z)(\xi_i^w + a_i^z\z)\\
				 &=& \xi^k_{i|k} + \Gamma_w\xi_k^w\xi_i^w + a_k^z\xi^w_{i|z} + a_k^z\Gamma_w\z\xi_i^k
				 + a^z_{i|k}\z + a_i^z \zeta^w_{z|k} + \Gamma_w \xi_k^wa_i\z \\
				 &+& a_k^za_{i|z}\z + a_k^za_i^z\zeta^w_{z|z} + a_k^z\Gamma_w\z a_i\z 	 
\end{eqnarray*}
Because of $a_k^za_{i|z}^z - a_i^za_{k|z}^z = \partial_k\log (g) a_i - \partial_i \log (g) a_k$ and
$a_{i|k}^z - a_{k|i}^z = \partial_i \log (g) a_k - \partial_k \log (g) a_i$
the assertion follows.
For the second and the fourth summand we prove the following proposition
\begin{proposition}
\[
\int_{X_s}{\varphi_{i\jbar}L_k(\beta_s^*(\wY))} = \int_{X_s}{u_k\cdot D_i(u_{\jbar})g dA}
\]
\end{proposition}
\begin{proof}
We write
\[
\hz=\hw\z\zb.
\]
It holds
\begin{eqnarray*}
L_k(\beta_s^*(\wY))_{z\zbar} &=& \left[\partial_k+a_k^z\partial_z,\hz\right] = h_{z\zbar|k} + 
a_k^z h_{z \zbar |z} + a_{k|z}^z\hz\\
										&=& h_{w\wbar|k}\z\zb + \hw\zeta^w_{z|k}\zb + a_k^zh_{w\wbar|z}\z\zb + a_k^z\hw\zeta^w_{z|z}\zb
										+ a_{k|z}^z\hz\\
										 &=& \Gamma_w\hw\xi_k^w\z\zb + \hw\zeta^w_{z|k}\zb + a_k^z\Gamma_w\hw\z\z\zb \\										
&+& a_k^z\hw\zeta^w_{z|z}\zb + a_{k|z}^z\hz.
\end{eqnarray*}
Moreover,
$$
\nabla_{v_i}(u_{\jbar})=(\nabla_i + a_i^z\nabla_z)(\xi_{\jbar}^{\wbar} + a_{\jbar}^{\zbar}\zb)= a_{\jbar|i}^{\zbar}\zb + a_iA_{\jbar z}^{\zbar}\zb = -\ginv \partial_z\varphi_{i\jbar}\zb.
$$
We rewrite the form $\hw \z \zb dz \wedge d \zbar$ as $\hw \zb dw \wedge d \zbar$ and contract it with the vector field $(\xi^w_k+a_k^z\z)\dl_w$. We obtain in this way that the tensor $\zb(\xi_k^w+a_k^z\z)\hw$ and thus also  $\varphi_{i\jbar}\zb(\xi_k^w+a_k^z\z)\hw$ is well-defined.
Stokes' theorem applied to this globally defined $(0,1)$-form $\varphi_{i\jbar}\zb(\xi_k^w+a_k^z\z)\hw$ yields
\[
\int{\nabla_z(\varphi_{i\jbar}\zb(\xi_k^w+a_k^z\z)\hw)dA}=0
\] 
and thus
\[
-\int{\partial_z\varphi_{i\jbar}\zb(\xi_k^w+a_k^z\z)\hw dA} = \int{\varphi_{i\jbar}\zb(\xi_k^w+a_k^z\z)_{;z}\hw dA}.
\]
Now
\[
(\xi_k^w+a_k^z\z)_{;z} = \xi_{k|z}^w + \Gamma_w\z\xi_k^w + a_{k|z}^z\z + a_k^z\zeta^w_{z|z} + 
\Gamma_w\z a_k^z\z , 
\]
so
\begin{eqnarray*}
(\xi_k^w+a_k^z\z)_{;z}\zb\hw &=& \xi_{k|z}^w\zb\hw + \Gamma_w\z\xi_k^w\zb\hw\\ &+& 
a_k^z\Gamma_w\z\hw\z\zb +  a_k^z\hw\zeta^w_{z|z}\zb + a_{k|z}^z\z\zb\hw \\
&=& L_k(\hw\z\zb)_{z\zbar}.
\end{eqnarray*}
This proves the proposition. 
\end{proof}
\section{Computation of the curvature}
The computation of the curvature of the Weil-Petersson metric on the base of a general effectively parametrized family 
$(\lb,f): \X \to Y \times S$ of coverings of Riemann surfaces seems to be difficult and leads to overflowed expressions with hardly any interpretation. The reason for this relies on the fact that in general there is no intimate relation between the hyperbolic metrics on the fibers $X_s$ and the metric on $Y$. But both hermitian metrics contribute to the expression for the Weil-Petersson metric, because it measures the variation of the complex structure on the fibers $X_s$ as well as the variation of the covering map $\lb_s: X_s \to Y$. It is in particular the term $G_0^{WP}$ which is quite difficult to deal with in this context. However, this term disappears if the underlying family of complex structures is locally trivial. We can give a curvature formula for such families. This covers the case of complex subspaces of the Hurwitz space $\Hbn$, which parametrize families of coverings $\lb_s: X_s \to \PG$ where the complex structure of $X_s$ is fixed. Moreover, we can give a curvature formula of the bundle $f_*\lb^*T_{\PG}$, which is then a bundle on the entire Hurwitz space and gives the curvature formula of the subspaces by restriction.  

\subsection{Curvature for a subspace}
We consider the universal family $(\beta,f): \X=(X_s)\to \PG \times \Hbn$ over the Hurwitz space. Let $b>4g-4$. By Serre duality and $\deg(K_{X_s}\otimes (\beta_s^*T_{\PG}))<0$, the space $H^1(X_s,\beta_s^*T_{\PG})$ is trivial. Thus we obtain the short exact sequence    
$$
0 \to H^0(X_s,\beta_s^*T_{\PG}) \to H^0(X_s,N_{\beta_s}) \to H^1(X_s,T_{X_s}) \to 0.
$$ 
This is the tangent sequence belonging to the map $\Hbn \to \M$. For this to be true, we have to restrict to the open part $\M^0$ which parametrizes Riemann surfaces with trivial automorphism group. Denoting the corresponding inverse image by $\Hbn_0$, we get a submersion $\Hbn_0 \to \M^0$. Alternatively, we can move to the universal covering $\tilde{\Hbn}$ and get a submersion $\tilde{\Hbn} \to \mathcal{T}_g$ onto the Teichm\"uller space.
A fixed compact Riemann surface $X$ of genus $g>1$ without non-trivial automorphisms represents the isomorphism class of a complex structure $[X]$, that is a point in $\M^0$. We denote the corresponding fiber under the submersion $\Hbn \to \M^0$ by 
$\Hx$. This subspace is a (maybe non connected) complex submanifold of dimension $r:=b-(3g-3)=2n-(g-1)$. By construction, the fibers of the restricted family
\[
f|_{f^{-1}(\Hx)} \to \Hx
\]
are all isomorphic and by a result of Grauert and Fischer, the family is complex analytic locally trivial. The points of this subspace now parametrize isomorphism classes of simple branched $(n,b)$-coverings where the surface $X$ (i.e. its complex structure) is fixed. Let $s_0 \in \Hx$ be given by a covering $\beta_0: X \to \PG$. We choose local holomorphic coordinates $s^1,\ldots,s^b$ so that the subspace $\Hx$ is locally given by 
\[
\{
s \in \Hbn \; | \; s^{r+1}=\cdots=s^b=0 
\},
\]
and thus we can take $s^1,\ldots,s^r$ as local coordinates on $\Hx$. Because our computations are local in the base, we can restrict our family to a possibly smaller base $S \subset \Hx$ and assume that the family $\X_S=X \times S$ is in fact trivial. We study the metric tensor $G_{i\jbar}^{WP}$ for the base $S$ and the family
$$
(\beta,f): X \times S \to \PG \times S.
$$
The family of complex structures is trivial, so in particular infinitesimal trivial. Hence by Corollary \ref{InfTriv} $\varphi_{i\jbar}(z,s)=0$  on $X \times S$ for $1 \leq i,j \leq r$. Thus the first summand $G_{0,i\jbar}^{WP}$ does not contribute to the metric. Furthermore, we have on $X\times S$ the constant family of metric tensors $g(z,s)=g(z)$, where $g$ is the hyperbolic metric of constant Ricci curvature $-1$ on $X$. The coordinate vector fields $(\dl_i)_{1 \leq i \leq r}$ hence exist on $X\times S$ and coincide with the horizontal lifts $(v_i)$ (that means $a_i=0$ for all $1\leq i\leq r$). The expression for the metric tensor thus reduces to
\[
G_{i\jbar}(s_0) = \int_{X} {\xi^w_i \xi^{\wbar}_{\jbar}\hw \; g dA}. 
\] 
We start computing by using again Lie derivatives:
\begin{eqnarray*}
\partial_k G_{i\jbar}(s_0) &=& \int_X{L_k(\xi^w_i \xi^{\wbar}_{\jbar}\hw \; g dA)} \\ 
&=& \int{\partial_k(\xi^w_i \xi^{\wbar}_{\jbar}\hw)\; g dA}
\end{eqnarray*}
Since
\begin{eqnarray}
\partial_k(\xi^w_i \xi^{\wbar}_{\jbar}\hw) = \xi^w_{i|k}\xi^{\wbar}_{\jbar}\hw 
+ \xi^w_i \xi^{\wbar}_{\jbar} \partial_k(h(\beta(z,s))) 
\label{kovAbl}
\end{eqnarray}
and
\begin{eqnarray*}
\partial_k(h(\beta(z,s))) &=& \frac{\dl h}{\dl w}\left(\beta(z,s)\right) \cdot 
\frac{\dl \beta}{\dl s^k} \\ &=& (\dl_wh)(\beta(z,s)) \xi^w_k \\
&=& (h(\beta(z,s)))^{-1} (\dl_wh)(\beta(z,s))\xi^w_k(z,s)h(\beta(z,s))\\
&=& \Gamma_w \xi^w_k \hw
\end{eqnarray*}
if we set $\Gamma_w = \beta^*\Gamma_h$ where $\Gamma_h = \dl_w \log (h) = (\dl_w h)h^{-1}$,
we identify \ref{kovAbl} as a covariant derivative and write 
\[
\partial_k(\xi^w_i \xi^{\wbar}_{\jbar}\hw) = (\nabla_k\xi^w_i)\xi^{\wbar}_{\jbar}\hw.
\]
Thus, as a first result we obtain
\begin{lemma}
\label{ersteAbl}
$$
\dl_k G_{i\jbar}(s_0) = \int_X{(\nabla_k \xi^w_i)\xi^{\wbar}_{\jbar}\hw \; g dA}
$$
\end{lemma}
We continue computing
\begin{eqnarray*}
\dl_{\lbar}\dl_k G_{i\jbar}(s_0) &=& \dl_{\lbar} \int_X{(\nabla_k\xi^w_i)\xi^{\wbar}_{\jbar}\hw \; g dA}\\
&=& \int_X{\dl_{\lbar}\left((\nabla_k\xi^w_i)\xi^{\wbar}_{\jbar}\hw \right)  g dA} 
\end{eqnarray*}
Analogously to above we get
\begin{eqnarray*}
\dl_{\lbar}\left((\nabla_k\xi^w_i)\xi^{\wbar}_{\jbar}\hw \right) &=& 
\dl_{\lbar}(\nabla_k \xi^w_i)\xi^{\wbar}_{\jbar}\hw + 
\nabla_k \xi^w_i \nabla_{\lbar}(\xi^{\wbar}_{\jbar})\hw
\end{eqnarray*}
where
\begin{eqnarray*}
\dl_{\jbar} (\nabla_k \xi^w_i) &=& \dl_{\lbar}\left(\xi^w_{i|k} + \Gamma_w \xi^w_k \xi^w_i\right)\\
&=& \dl_{\lbar}\Gamma_w \xi^w_k \xi^w_i
\end{eqnarray*}
But now
\begin{eqnarray*}
\dl_{\lbar} \Gamma_w &=& \dl_{\lbar} \left( h(\beta(z,s))^{-1} (\dl_wh)(\beta(z,s)) \right) \\
&=& - (\dl_{\wbar}h) \left( \beta(z,s)\right)\xi^{\wbar}_{\lbar}(\dl_w h)(\beta(z,s))h(\beta(z,s))^{-2}\\ &+& h(\beta(z,s))^{-1}(\dl_{\wbar}\dl_w h)(\beta(z,s))\xi^{\wbar}_{\lbar}. 
\end{eqnarray*}
Since
\begin{eqnarray*}
\dl_{\wbar} \dl_w \log h &=& \dl_{\wbar} \left((\dl_w h)h^{-1}\right)\\
&=& (\dl_{\wbar} \dl_w h)h^{-1} - (\dl_w h)(\dl_{\wbar})h^{-2}\\
&=& -K_h,
\end{eqnarray*}
we have $\dl_{\lbar}\Gamma_w = -K_{w\wbar} \xi^{\wbar}_{\lbar}$, where we set
\[
K_{w\wbar}:=\beta^*(K_h)= \beta^*(-\dl_{\wbar}\dl_w \log h) = K_{\PG} \hw.
\]
(Of course we have $K_{\PG}=1$, but we prefer to write $K_{\PG}$ for keeping track of the influence of $K_{\PG}$.) 
Altogether we obtain
\begin{lemma}
\[
\dl_{\lbar}\dl_k \G_{i\jbar}(s_0) = 
\int_X{(\nabla_k \xi^w_i)(\nabla_{\lbar}\xi^{\wbar}_{\jbar})\hw \; g dA}
- K_{\PG} \int_X{\xi^w_i\xi^{\wbar}_{\jbar}\xi^w_k \xi^{\wbar}_{\lbar}\hw^2 \; g dA}
\]
\end{lemma}
Now we choose normal coordinates
around the point $s_0$, that means coordinates such that
\[
\dl_k G_{i\jbar}(s_0)= 0 = \dl_{\lbar}G_{i\jbar}(s_0),
\]
that is
\[
\int_X{(\nabla_k \xi^w_i)\xi^{\wbar}_{\jbar})\hw \; g dA} = 0.
\]
Since the vector fields $\{\xi^w_k\}\dl_w$ for $1 \leq k \leq r$ form a base of $H^0(X,\beta_0^*(T_Y))$, this means that
\[
\left( \nabla_k\xi^w_i \right) \bot \; H^0(X,\beta_0^*(T_Y)).
\]
Since $\laplacedbar = \dbarstar \dbar$ on the space of differentiable vector fields with values in $\beta_0^*(T_Y)$ and because of the identity $\id = \Hbar + \Gbar \laplacedbar$, we can now write
\begin{eqnarray*}
\int_X{(\nabla_k \xi^w_i)(\nabla_{\lbar}\xi^{\wbar}_{\jbar})\hw \; g dA} &=& 
\left\langle \nabla_k\xi^w_i, \nabla_l \xi^w_j \right\rangle \\
&=& \left\langle \Gbar \laplacedbar \left( \nabla_k\xi^w_i \right), \nabla_l \xi^w_j \right\rangle \\
&=& \left\langle \dbarstar \Gbar \dbar \left( \nabla_k\xi^w_i \right), \nabla_l \xi^w_j \right\rangle \\
&=& \left\langle \Gbar \dbar \left( \nabla_k\xi^w_i \right), \dbar \nabla_l \xi^w_j \right\rangle
\end{eqnarray*}
Since
\[
\nabla_k \xi^w_i = \xi^w_{i|k} + \Gamma_w \xi^w_k \xi^w_i,
\]
we have
\[
\dbar\left( \nabla_k \xi^w_i \right) = - \Kw \xi^w_k \xi^w_i \zb d\zbar,
\]
and thus
\begin{eqnarray*}
\int_X{(\nabla_k \xi^w_i)(\nabla_{\lbar}\xi^{\wbar}_{\jbar})\hw \;  g dA} &=& 
\left\langle \Gbar \dbar \left( \nabla_k\xi^w_i \right), \dbar \nabla_l \xi^w_j \right\rangle \\
&=& K_{\PG}^2 \left\langle  \Gbar(\xi^w_i \xi^w_k \zb \hw d\zbar), \xi^w_j \xi^w_l \zb \hw d\zbar \right\rangle
\end{eqnarray*}
Introducing for abbreviating
\[
\psi_{ik} d\zbar \otimes \dl_w := \left( \xi^w_i \xi^w_k \zb \hw \right) d\zbar \otimes \dl_w, 
\]
we can now write
\begin{eqnarray*}
\int_X{(\nabla_k \xi^w_i)(\nabla_{\lbar}\xi^{\wbar}_{\jbar})\hw \;  g dA} &=&
K_{\PG}^2 \int_X{\Gbar \left( \psi_{ik} \right)^w_{\zbar} (\psi_{\jbar\lbar})^{\wbar}_z \hw \; dz\wedge d\zbar } \\
&=& K_{\PG}^2 \int_X{\Gbar \left( \psi_{ik} \right) \cdot \psi_{\jbar\lbar} \; g dA}.
\end{eqnarray*}
We arrive at
\begin{theorem}
\label{kruemmungUnterraum}
The curvature of the Weil-Petersson metric for the subspace $\Hx \subset \Hbn$ is given by
\begin{eqnarray*}
R_{i\jbar k \lbar}(s_0) = &-&  \int_X{\Gbar \left( \psi_{ik} \right) \cdot \psi_{\jbar\lbar}\; g dA} \\
                          &+& \int_X{(\xi_i \cdot \xi_{\jbar}) \, (\xi_k \cdot \xi_{\lbar}) \; g dA}  
\label{kr"ummung}
\end{eqnarray*}
where $\psi_{ik} d\zbar \otimes \dl_w := \left( \xi^w_i \xi^w_k \zb \hw \right) d\zbar \otimes \dl_w.$
\end{theorem}

\subsection{Curvature of $f_*\beta^*T_{\PG}$}
We consider the universal family of simple $(n,b)$-coverings over the Hurwitz space $(\beta,f): \X \to \PG \times \Hbn$ for $b>4g(X_s)-4$. By Serre duality and $\deg(K_{X_s}\otimes (\beta_s^*T_Y)^*)<0$ we get $h^1(X_s,\beta_s^*T_{\PG})=0$ for all $s \in \Hbn$ in this case. Hence the dimension of $H^0(X_s,\beta_s^*T_{\PG})$ is constant on $\Hbn$ and different from $0$ by the theorem of Riemann-Roch. Thus, the sheaf $f_*\beta^*T_{\PG}$ is locally free on $\Hbn$, that is a vector bundle. We denote its rang by 
$r:=h^0(X_s,\beta_s^*T_{\PG})$. The fiber of this bundle over $s \in \Hbn$ is just the space $H^0(X_s,\beta_s^*T_{\PG})$. For a neighborhood $U\subset \Hbn$, we identify the space of holomorphic sections $\Gamma(U,f_*\beta^*T_{\PG})$ over $U$ with 
$\Gamma(f^{-1}(U),\beta^*T_{\PG})$, that is the space of holomorphic vector fields on $f^{-1}(U)=:\X_U$ with values in the pullback of the tangent bundle $\beta^*T_{\PG}$. For a base of local holomorphic sections $u_1, \ldots, u_r \in H^0(\X_U,\beta^*T_{\PG})$, the natural $L^2$-metric is given by 
$$
G_{i \jbar}(s)=\int_{X_{s}}u_i(z,s) \ovl{u_j(z,s)}h(\beta(z,s))g(z,s)dA .
$$
This metric coincides with the induced Weil-Petersson metric $G_1^{WP}$ on the subbundle of the tangent bundle
$$
f_*\beta^*T_{\PG} \subset T_{\Hbn}.
$$ 
\begin{remark}
The bundle $f_*\beta^*T_Y$ and the corresponding metric are also defined for a family $(\beta,f): \X \to Y \times S$ when $Y$ is a torus. In this case $T_Y$ is trivial. Hence also $\beta^*T_Y$ is trivial on $\X$ and thus $f_*\beta^*T_Y$ is the trivial line bundle $\Oh_S$ on the base $S$, because $f$ has connected fibers. Taking a global trivialising section $u\equiv 1$ on $S$, we obtain a constant metric on the bundle $S \times \C$. For the case $g(Y)>1$ we have $h^0(X_s,\beta_s^*T_Y)=0$, since $\deg\beta^*T_Y<0$. In this case, the direct image sheaf $R^1f_*\beta^*T_Y$ is always a vector bundle on $S$.
\end{remark}
We start computing the curvature of the bundle $f_*\beta^*T_{\PG}$ on the base $\Hbn$ of dimension $b>4g(X_s)-4$. For this, we choose local coordinates $s^1,\ldots,s^b$ on $\Hbn$ in a neighborhood of a fixed point $s$ and an orthogonal frame
$u_1,\ldots,u_r$ of local holomorphic sections such that
$$
G_{i\jbar}(s)=\delta_{i\jbar}
$$
and
$$
\dl_kG_{i\jbar}(s)=0
$$
for all $1 \leq k \leq b$ und $1 \leq i,j \leq r$. 

We start computing and will use $L_k(g\, dA)=0$ (see Lemma \ref{LieDer}) several times. We get for the first derivative 
$$
\dl_k G_{i \jbar} (s)= \int_{X_s}{L_k(u_i^wu_{\jbar}^{\wbar}\hw \; gdA)}= \int_{X_s}{D_k(u_i^w)u_{\jbar}^{\wbar}\hw \; gdA},
$$
where $D_k= \nabla_k + a_k^z \nabla_z$ and $a_k^z$ is now in general only differentiable. Here we used that $u_{\jbar}$ is anti-holomorphic in $z$ and $s=(s^1,\ldots,s^b)$. For the second derivative, we obtain
\begin{eqnarray*}
\dl_{\lbar} \dl_k G_{i \jbar}(s) &=& \int_{X_s}{D_k(u_i^w)D_{\lbar}(u_{\jbar}^{\wbar})\hw \; gdA} \\
&+& \int_{X_s}{D_{\lbar}D_k(u_i^w)u_{\jbar}^{\wbar}\hw \; gdA}
\end{eqnarray*}
where
\begin{eqnarray*}
D_k(u_i) &=& u_{i;k} + a_k^zu_{i;z} \\
&=& u_{i,k} + \Gamma_w \xi^w_ku_i + a_k^z(u_{i,z} + \Gamma_w \z u_i)
\end{eqnarray*}
Since $a_k$ is in general not holomorphic,  we have
\begin{eqnarray*}
D_{\lbar}D_k(u_i) &=& (\dl_{\lbar} + a_{\lbar}^{\zbar}\dl_{\zbar})D_ku_i \\
&=& -\Kw\xi^{\wbar}_{\lbar}\xi^w_ku_i - a_{\lbar}^{\zbar}\Kw\zb\xi^w_ku_i \\
&-&(A_{k\zbar}^z a_{\lbar}^{\zbar} + g^{-1}\dl_{\zbar}\varphi_{k\lbar}) u_{i;z}\\
&+& a_{\lbar}^{\zbar}A_{k\zbar}^z u_{i;z} \\
&-& a_k^z \Kw \xi^{\wbar}_{\lbar}\z u_i - a_{\lbar}^{\zbar}a_k^z\Kw\zb\z u_i \\
&=& -\Kw (\xi^{\wbar}_{\lbar} + a_{\lbar}^{\zbar}\zb)(\xi^w_k + a_k^z \z)u_i \\
&-& g^{-1}\dl_{\zbar}\varphi_{k\lbar} u_{i;z}.
\end{eqnarray*}
Here we used equation $(4.5)$ of Lemma \ref{Liu}.
Altogether we obtain
\begin{eqnarray*}
-\dl_{\lbar} \dl_k G_{i\jbar}^{WP,1} (s)= &-& \int_{X_s}{D_k(u_i^w)D_{\lbar}(u^{\wbar}_{\jbar})\hw \; gdA} \\
&+&  K_{\PG}\int_{X_s}{(\xi^w_k + a_k^z \z)(\xi^{\wbar}_{\lbar} + a_{\lbar}^{\zbar}\zb)u_i^w u_{\jbar}^{\wbar}\hw^2 \; gdA} \\
&+& \int_{X_s}{g^{-1}\dl_{\zbar}\varphi_{k\lbar} u_{i;z}^w u_{\jbar}^{\wbar}\hw \; gdA}.
\end{eqnarray*}
Using Stokes' theorem, the third term can be rewritten as
\begin{eqnarray*}
&& \int_{X_s}{g^{-1}\dl_{\zbar}\varphi_{k\lbar} u_{i;z}^w u_{\jbar}^{\wbar}\hw \; gdA}\\
=&-& \int_{X_s}{g^{-1} \dl_z\dl_{\zbar}\varphi_{k\lbar} u_{i}^w u_{\jbar}^{\wbar}\hw \; gdA}\\
=&& \int_{X_s}{\Box\varphi_{k\lbar} u_{i}^w u_{\jbar}^{\wbar}\hw \; gdA}\\
\end{eqnarray*}
or alternatively as
\begin{eqnarray*}
&& \int_{X_s}{g^{-1}\dl_{\zbar}\varphi_{k\lbar} u_{i;z}^w u_{\jbar}^{\wbar}\hw \; gdA}\\
=&& \int_{X_s}{g^{-1} \varphi_{k\lbar}\Kw \z \zb u_{i}^w u_{\jbar}^{\wbar}\hw \; gdA}\\
&-& \int_{X_s}{g^{-1}\varphi_{k\lbar} u_{i;z}^w u_{\jbar;\zbar}^{\wbar}\hw \; gdA}\\
=&& K_{\PG} \int_{X_s}{\varphi_{k\lbar}u_{i}^w u_{\jbar}^{\wbar} \hw \; \beta_s^*(\omega_Y)}\\
&-& \int_{X_s}{\varphi_{k\lbar} (u_{i;z}^w/\z) (u_{\jbar;\zbar}^{\wbar}/\zb) \; \beta_s^*(\omega_Y)}\\
\end{eqnarray*}
\begin{theorem}
\label{KrBdl}
The curvature of the Weil-Petersson metric on the subbundle $f_*\beta^*T_{\PG} \subset T_{\Hbn}$ is given by
\begin{eqnarray*}
R^{WP}_{i\jbar k\lbar}(s)=&-& \int_{X_s}{D_k(u_i^w)D_{\lbar}(u^{\wbar}_{\jbar})\hw \; gdA} \\
&+&  \int_{X_s}{u_i^w u_{\jbar}^{\wbar}(\xi^w_k + a_k^z \z)(\xi^{\wbar}_{\lbar} + a_{\lbar}^{\zbar}\zb)\hw^2 \; gdA} \\
&+&  \int_{X_{s}}{\Box\varphi_{k\lbar} u_{i}^w u_{\jbar}^{\wbar}\hw \; gdA}
\end{eqnarray*}
\end{theorem}
The last term can be rewritten by using Stokes' theorem as
\begin{eqnarray*}
&& \int{\Box\varphi_{k\lbar} u_{i}^w u_{\jbar}^{\wbar}\hw \; gdA}\\
&=&- \int{g^{-1} \dl_z\dl_{\zbar}\varphi_{k\lbar} u_{i}^w u_{\jbar}^{\wbar}\hw \; gdA}\\
&=& \int{g^{-1}\dl_{\zbar}\varphi_{k\lbar} u_{i;z}^w u_{\jbar}^{\wbar}\hw \; gdA}\\
&=& \int{g^{-1} \varphi_{k\lbar}\Kw \z \zb u_{i}^w u_{\jbar}^{\wbar}\hw \; gdA}
- \int{g^{-1}\varphi_{k\lbar} u_{i;z}^w u_{\jbar;\zbar}^{\wbar}\hw \; gdA}\\
&=& K_{\PG} \int{\varphi_{k\lbar}u_{i}^w u_{\jbar}^{\wbar} \hw \; \beta_s^*(\omega_Y)}
- \int{\varphi_{k\lbar} (u_{i;z}^w/\z) (u_{\jbar;\zbar}^{\wbar}/\zb) \; \beta_s^*(\omega_Y)}\\
\end{eqnarray*}
Using the identity $\id=H +\laplacedbar \Gbar=\dbar^*\Gbar\dbar$, the first term can be split up  as
\begin{eqnarray*}
&-& \int{D_k(u_i^w)D_{\lbar}(u^{\wbar}_{\jbar})\hw \; gdA} \\
= &-& K_{\PG}^2 \int{\Gbar(u_i^w(\xi^w_k + a_k^z \z)\zb \hw)(u_{\jbar}^{\wbar} (\xi^{\wbar}_{\lbar} + a_{\lbar}^{\zbar}\zb)\z\hw)\hw\; dA} \\
&-& K_{\PG} \int{\Gbar(u_{i;z}^wA_{k\zbar}^z)(u_{\jbar}^{\wbar} (\xi^{\wbar}_{\lbar} + a_{\lbar}^{\zbar}\zb)\z\hw)\hw\; dA} \\
&-& K_{\PG} \int{\Gbar(u_i^w(\xi^w_k + a_k^z \z)\zb \hw)(u_{\jbar;\zbar}^{\wbar}A_{\lbar z}^{\zbar})\hw\; dA} \\
&-& \int{\Gbar(u_{i;z}^wA_{k\zbar}^z)(u_{\jbar;\zbar}^{\wbar}A_{\lbar z}^{\zbar})\hw\; dA}
\end{eqnarray*}

\begin{remark}
In his article \cite{Be09}, Berndtsson computes the curvature of the bundle $f_*(L\otimes K_{\X/S})$ for a general holomorphic fibration (submersion) $f: \X \to S$ and a line bundle $L$ on the K\"ahler manifold $\X$. He proves that the direct image 
$f_*(L \otimes K_{\X/S})$ is Nakano (semi-)positive if $L$ is (semi-)positive on $\X$ (compare also \cite{LSY13,LY14}). For the Hurwitz space $\Hbn$, the bundle
$L:=\beta^*T_{\PG} \otimes K_{\X/S}^{-1}$ is positive along the fibers for $b>4g-4$. But fiberwise,  the curvature of the metric $\beta^*h \cdot g$ on the line bundle $\beta^*T_{\PG} \otimes K_{\X/S}^{-1}$ is not positive everywhere, because it is negative in the ramification points.
\end{remark}


\begin{thebibliography}{[Hat02]}
\bibitem[ABS15]{ABS15}{\scshape R.~Axelsson, I.~Biswas, G.~Schumacher}: {\em K\"ahler Structure on Hurwitz spaces}, Manuscripta Math. \textbf{147}, 63-79 (2015).
\bibitem[ACGH85]{ACGH85}{\scshape E.~Arbarello, M.~Cornalba, P.A.~Griffiths, J.~Harris}: {\em Geometry of Algebraic Curves}, Vol. I, Grundlehren der math. Wiss., 267, Springer-verlag, New York (1985).
\bibitem[ACG11]{ACG11}{\scshape E.~Arbarello, M.~Cornalba, P.A.~Griffiths}: {\em Geometry of Algebraic Curves}, Vol.II, Grundlehren der math. Wiss., 268, Berlin: Springer (2011).
\bibitem[Be09]{Be09}{\scshape B.~Berndtsson}: {\em Curvature of vector bundles associated to holomorphic fibrations}, Ann. Math. \textbf{169}, 531-560 (2009).
\bibitem[BF02]{BF02}{\scshape P.~Bailey, M.~Fried}: {\em Hurwitz monodromy, spin separation and higher levels of modular tower}, 
Proc. of Symp. in Pure Math., Vol 70, Amer. Math. Soc., Providence, RI, 79-220 (2002).
\bibitem[Ca08]{Ca08}{\scshape A.~Cadoret}: {\em Lifting results for rational points on Hurwitz moduli spaces}, Israel Journal of Math. \textbf{164}, 19-59 (2008).
\bibitem[Cl72]{Cl72}{\scshape A.~Clebsch}: {\em Zur Theorie der Riemannschen Fl\"achen}, Math. Ann. \textbf{6}, 216-230 (1872).
\bibitem[DF99]{DF99}{\scshape P.~D\`{e}bes, M.~Fried}: {\em Integral specialization of families of rational functions}, PJM \textbf{190}, 45-85 (1999).
\bibitem[Fr12]{Fr12}{\scshape S.~Francaviglia}: {\em Relations between the complex structure of Hurwitz spaces,
the cut-and-paste topology, and the complex structure induced by a BPS} http://www.dm.unibo.it/~francavi/lavori/cstrbps.pdf  (2012).
\bibitem[Fri77]{Fri77}{\scshape M.~Fried}: {\em Fields of definition of function fields and Hurwitz families - groups as Galois groups}, Comm. Alg. \textbf{5}, 17-82 (1977).
\bibitem[Fu69]{Fu69}{\scshape W.~Fulton}: {\em Hurwitz schemes and irreducibility of moduli of algebraic curves}, Ann. Math. \textbf{90}, 542-575 (1969).
\bibitem[HM98]{HM98}{\scshape J.~Harris, I.~Morrison}: {\em Moduli of curves}, Graduate Texts in Math., Vol 187, New York:  Springer (1998).
\bibitem[HM82]{HM82}{\scshape J.~Harris, D.~Mumford}: {\em On the Kodaira Dimension of the Moduli Space of Curves}, Invent. Math. \textbf{67}, 23-86 (1982).
\bibitem[HGS02]{HGS02}{\scshape J.~Harris, T.~Graber, J.~Starr}: {\em A note on Hurwitz schemes of covers of a positive genus curve}, arxiv:math/0205056v1 (2002).
\bibitem[Ho73]{Ho73}{\scshape E.~Horikawa}: {\em On deformations of holomorphic maps I}, J. Math. Soc. Japan \textbf{25}, 372-396 (1973).
\bibitem[Ho74]{Ho74}{\scshape E.~Horikawa}: {\em On deformations of holomorphic maps II}, J. Math. Soc. Japan \textbf{26}, 647-667 (1974).
\bibitem[Hu91]{Hu91}{\scshape A.~Hurwitz}: {\em \"Uber Riemannsche'sche Fl\"achen mit gegebenen Verzweigungspunkten}, Math. Ann. \textbf{39}, 1-61 (1891).
\bibitem[Ko70]{Ko70}{\scshape S.~Kobayashi}: {\em Hyperbolic manifolds and holomorphic mappings}, Dekker (1970).
\bibitem[LSY04]{LSY04}{\scshape K.~Liu, X.~Sun, S.-T.~Yau}: {\em Canonical metrics on the moduli space of Riemann surfaces I}, J. Differential Geom. \textbf{68}, No. 3, 571-637 (2004).
\bibitem[LSY13]{LSY13}{\scshape K.~Liu, X.~Sun, X.~Yang}: {\em Vanishing theorems for ample vector bundles}, J. Algebraic Geom. \textbf{22}, 303-331 (2013).
\bibitem[LY14]{LY14}{\scshape K.~Liu, X.~Yang}: {\em Curvature of direct image sheaves of vector bundles and applications}, J. Differential Geom. \textbf{98}, No. 1, 117-145 (2014).
\bibitem[Ma05]{Ma05}{\scshape M.~Manetti}: {\em Lectures on deformations of complex manifolds}, arXiv:math/0507286v1 (2005).
\bibitem[Na16]{Na16}{\scshape P.~Naumann}: {\em K\"ahlersche Geometrie auf Hurwitz-R\"aumen}, Dissertation, Marburg (2016).
\bibitem[Na79]{Na79}{\scshape M.~Namba}: {\em Families of Meromorphic Functions on Compact Riemann Surfaces}, Lecture Notes in Math., Vol. 767, Springer (1979).
\bibitem[NR11]{NR11}{\scshape T.~Napier, M.~Ramachandran}: {\em An Introduction to Riemann Surfaces}, Springer (2011).
\bibitem[Pa13]{Pa13}{\scshape A.P.~Patel}: {\em The geometry of the Hurwitz space}, Dissertation, Harvard (2013).
\bibitem[Ri57]{Ri57}{\scshape B.~Riemann}: {\em Theorie der Abel'schen Functionen}, Borchardt's Journal f"ur reine und angewandte Mathematik, Bd 54 (1857).
\bibitem[Sch93]{Sch93}{\scshape G.~Schumacher}: {\em The curvature of the Petersson-Weil metric on the moduli space of 
K\"ahler-Einstein manifolds}, Ancona, V. (ed.) et al., Complex analysis and geometry. New York: Plenum Press. The University Series in Math., 339-354 (1993).
\bibitem[Sch12]{Sch12}{\scshape G.~Schumacher}: {\em Positivity of relative canonical bundles and applications}, Invent. Math. \textbf{190}, 1-56 (2012).
\bibitem[Se06]{Se06}{\scshape E.~Sernesi}: {\em Deformations of Algebraic Schemes}, Grundlehren der math. Wiss., Vol. 334, Berlin: Springer (2006)
\bibitem[Sev21]{Sev21}{\scshape F.~Severi}: {\em Vorlesungen \"uber algebraische Geometrie}, Teubner-Verlag (1921).
\bibitem[Siu86]{Siu86}{\scshape Y.T.~Siu}: {\em Curvature of the Weil-Petersson metric in the moduli space of compact K\"ahler-Einstein manifolds of negative first Chern class}, In: H.W: Stoll (ed.) Contributions to Several Complex Variables, Proc. Conf. Complex Analysis, Norte Dame/Indiana, 1984. Aspects Math., vol. E9, pp. 261-298 (1986) 
\end{thebibliography}
\end{document}